\documentclass[12pt]{amsart}

\usepackage{bm,latexsym,amsfonts,amsmath,amssymb,amsthm,url,graphics,psfrag,amscd,stmaryrd, mathrsfs}
\usepackage[english]{babel}
\usepackage{todonotes}
\usepackage[T1]{fontenc}

\usepackage{graphics}
\usepackage{graphicx}
\usepackage{color}

\newcommand{\R}{\mathbb{R}}

\newcommand{\N}{\mathbb{N}}
\newcommand{\E}{\mathbb{E}}

\newcommand{\pp}{\mathbb{P}}

\newcommand{\kA}{\mathcal{A}}

\newcommand{\kO}{O}
\newcommand{\kP}{\mathcal{P}}

\newcommand{\kM}{\mathcal{M}}

\newcommand{\bq}{\boldsymbol{q}}
\newcommand{\bp}{\boldsymbol{p}}

\newcommand{\lin}{\left[\kern-0.15em\left[}
\newcommand{\rin} {\right]\kern-0.15em\right]}
\newcommand{\linf}{[\kern-0.15em [}
\newcommand{\rinf} {]\kern-0.15em ]}
\newcommand{\ilin}{\left]\kern-0.15em\left]}
\newcommand{\irin} {\right[\kern-0.15em\right[}
\newcommand{\sqchb} {\sqrt{\cosh(\beta)}}

\newtheorem{theorem}{Theorem}[section]

\newtheorem{lemma}[theorem]{Lemma}

\newtheorem{proposition}[theorem]{Proposition}

\newtheorem{remark}[theorem]{Remark}
\newtheorem{cond}[theorem]{Condition}
\newtheorem{example}[theorem]{Example}

\newcommand{\be}{\begin{equation}}
\newcommand{\ee}{\end{equation}}
\newcommand{\nn}{\nonumber}

\newcommand{\prob}{\mathbb P}
\newcommand{\expec}{\mathbb E}

\newcommand{\atanh}{{\rm atanh}}

\newcommand{\sss}{\scriptscriptstyle}

\numberwithin{equation}{section}
\newcommand{\vep}{\varepsilon}
\newcommand{\e}{{\rm e}}

\newcommand{\rem}[1]{}

\newcommand{\RvdH}[1]{\todo[inline, color=magenta]{#1}}
\newcommand{\Gib}[1]{\todo[inline, color=yellow]{#1}}

\def\1{{\mathchoice {1\mskip-4mu\mathrm l}      % Blackboard bold 1
{1\mskip-4mu\mathrm l}
{1\mskip-4.5mu\mathrm l} {1\mskip-5mu\mathrm l}}}
\newcommand{\indic}[1]{\1_{\{#1\}}}

\newcommand{\col}[1]{\textcolor[rgb]{0,0,0}{#1}}

\newcommand{\cla}[1]{\textcolor[rgb]{0,0,0}{#1}}
\newcommand{\claudio}[1]{\textcolor[rgb]{0,0,1}{#1}}
\newcommand{\colrev}[1]{\textcolor[rgb]{0,0,0}{#1}}
\newcommand{\mo}[1]{\textcolor[rgb]{0,0,0}{#1}}

\newcommand{\eqn}[1]{\begin{equation}#1\end{equation}}
\newcommand{\eqan}[1]{\begin{align}#1\end{align}}

\newcommand{\convd}{\stackrel{\sss {\mathcal D}}{\longrightarrow}}

\newcommand{\psinand}{\psi_n^{\rm \sss an, d}}
\newcommand{\psinaniid}{\psi_n^{\rm \sss an, D}}

\newcommand{\varphiand}{\varphi^{\rm \sss an, d}}
\newcommand{\Mand}{M^{\rm \sss an, d}}
\newcommand{\Msand}{M^{\rm \sss an, d}}
\newcommand{\varphianiid}{\varphi^{\rm \sss an, D}}
\newcommand{\Maniid}{M^{\rm \sss an, D}}
\newcommand{\Msaniid}{M^{\rm \sss an, D}}

\newcommand{\betacand}{\beta_c^{\rm \sss an, d}}
\newcommand{\betacaniid}{\beta_c^{\rm \sss an, D}}

\newcommand{\betacqud}{\beta_c^{\rm \sss qu, d}}
\newcommand{\betacquiid}{\beta_c^{\rm \sss qu, D}}

\newcommand{\bfdit}{\boldsymbol{d}}

\newcommand{\CMnd}{{\rm CM}_n(\boldsymbol{d})}
  
\newcommand{\Rbold}{{\mathbb R}}
\newcommand{\Ver}{U_n}
\newcommand{\pkn}{p_k^{\sss(n)}}
\newcommand{\dmax}{d_{\max}^{(n)}}
\newcommand{\dmin}{d_{\min}}

\newcommand{\bsp}{\boldsymbol{p}}
\newcommand{\bsq}{\boldsymbol{q}}
\newcommand{\bsqs}{\boldsymbol{q}^\star}
\newcommand{\bst}{\boldsymbol{t}}

\newcommand{\bsqn}{\boldsymbol{q}^{\sss(N)}}
\newcommand{\qn}{q^{\sss(N)}}
\newcommand{\sn}{s^{\sss(N)}}
\newcommand{\qs}{q^{\star}}
\newcommand{\changed}[1]{\color{black}#1\color{black}}

\usepackage{fancyhdr}
\begin{document}

\title{\col{Annealed Ising model on configuration models}}
\author{Van Hao Can}
\address{Van Hao Can, Research Institute for Mathematical Sciences, Kyoto University, 606--8502 Kyoto, Japan 
 \&	Institute of Mathematics, Vietnam Academy of Science and Technology, 18 Hoang Quoc Viet, 10307 Hanoi, Vietnam, cvhao89@gmail.com}
\author{Cristian Giardin\`a}
\address{Cristian Giardin\`a, University of Modena and Reggio Emilia, via G. Campi 213/b, 41125 Modena, Italy, cristian.giardina@unimore.it}
\author{Claudio Giberti}
\address{Claudio Giberti, University of Modena and Reggio Emilia, via G. Amendola 2, Pad. Morselli, 42122 Reggio E., Italy, claudio.giberti@unimore.it}
\author{Remco van der Hofstad}
\address{Remco van der Hofstad, Eindhoven University of Technology, P.O. Box 513, 5600 MB Eindhoven, The Netherlands, r.w.v.d.hofstad@tue.nl}

%\begin{abstract}
%In this paper, we study the annealed ferromagnetic Ising model on the configuration model. In an annealed system, we take the average on both sides of the ratio {defining the Boltzmann-Gibbs measure of the Ising model}. In the configuration model, the degrees are specified. Remarkably, when the degrees are deterministic, the critical value 
%of the annealed Ising model is the same as that for the quenched Ising model. For independent and identically distributed (i.i.d.) degrees, instead, the annealed critical value is strictly smaller than that of the quenched Ising model. This identifies the degree structure of the underlying graph as the main driver for the critical value. 
%Furthermore, in both contexts (deterministic or random degrees), we provide the variational expression for the annealed pressure.
%Interestingly, our rigorous results establish that only part of the heuristic conjectures in the physics literature 
%were correct.
%\end{abstract}

\maketitle

\section*{Abstract}
In this paper, we study the annealed ferromagnetic Ising model on the configuration model. In an annealed system, we take the average on both sides of the ratio {defining the Boltzmann-Gibbs measure of the Ising model}. In the configuration model, the degrees are specified. Remarkably, when the degrees are deterministic, the critical value 
of the annealed Ising model is the same as that for the quenched Ising model. For independent and identically distributed (i.i.d.) degrees, instead, the annealed critical value is strictly smaller than that of the quenched Ising model. This identifies the degree structure of the underlying graph as the main driver for the critical value. 
Furthermore, in both contexts (deterministic or random degrees), we provide the variational expression for the annealed pressure.
Interestingly, our rigorous results establish that only part of the heuristic conjectures in the physics literature 
were correct.

\vspace{1.cm}

\section*{R\'esum\'e}
Dans ce travail, nous \'etudions le mod\`ele d'Ising ferromagn\'etique ``annealed'' sur le mod\`ele de configuration. Dans un syst\`eme ``annealed'', nous prenons la moyenne des deux c\^ot\'es du rapport {d\'efinissant la mesure de Boltzmann-Gibbs du mod\`ele d'Ising}. Dans le mod\`ele de configuration, les degr\'es sont sp\'ecifi\'es. Remarquablement, lorsque les degr\'es sont d\'eterministes, la valeur critique
du mod\`ele d'Ising ``annealed'' est le m\^eme que celui du mod\`ele d'Ising ``quenched''. Pour des degr\'es ind\'ependants et distribu\'es de mani\`ere identique (i.i.d.), au contraire, la valeur critique ``annealed'' est strictement inf\'erieure \`a celle du mod\`ele d'Ising ''quenched''. Cela identifie la structure en degr\'es du graphique sous-jacent comme le principal moteur de la valeur critique.
Dans les deux contextes (degr\'es d\'eterministes ou al\'eatoires), nous fournissons l'expression variationnelle de la pression ``annealed''.
Fait int\'eressant, nos r\'esultats rigoureux \'etablissent que seule une partie des conjectures heuristiques de la litt\'erature de physique
\'etaient corrects.

\vspace{1.5cm}

{\em Keywords}: Annealed Ising model, configuration model, phase transition, critical temperature.

\newpage

%\RvdH{Suggestions and to be discussed:
%\begin{itemize}
%\item[$\rhd$] Do we keep examples in this paper, or only in more applied paper?
%\item[$\rhd$] Should we discuss the GRG here? I propose to remove this bit completely. It is not in the title, we do not yet know how critical values behave, and it seems much more like an open end. 
%It might be better to include the math part in the more applied paper.
%\item[$\rhd$] Remaining minor issues.
%\item[$\rhd$] What to include in the applied paper?
%\end{itemize}
%}

\section{Motivation and results}
\label{sec-mot-res}
The Ising model is a paradigmatic model for magnetism, see \cite{Niss05, Niss09} for an extensive historical account, and the books by Ellis \cite{Elli85}, Bovier \cite{Bovi06} and Contucci and Giardin\`a \cite{ConGia13} for introductions to statistical mechanics from various perspectives.   Recently, the Ising model has gained in popularity since it gives a very simple model for {\em consensus reaching} in populations. Indeed, when we assume that friends are more likely to have the same opinion rather than opposing ones, then we can use the Ising model on the friendship network to model which opinion will prevail in a population and under which circumstances. As a result, the model has become popular amongst economists and social scientists as well, and later found further applications in neuroscience, etc. See for example Contucci, Gallo and Menconi \cite{ContGallMen08} for an example in social sciences, Kohring \cite{Kohr96} for an application in social impact, Fraiman, Balenzuela, Foss and Chialvo \cite{FraBalFosChi09} for an application to the brain and Bornholdt and Wagner \cite{BorWag02} for an example in economics.
Whereas the original Ising model was defined on a regular lattice, in recent applications more complex spatial structures (often in the form of a random graph) have been
considered to model system inhomogeneities.

Being a paradigmatic model for random variables that are dependent, the Ising model \mo{on random graphs} has also attracted considerable attention from the mathematics community, starting with Dembo and  Montanari \cite{DemMon10a}, see also \cite{DemMon10b, DemMonSun13} for an overview of statistical mechanics models on random graphs, as well as \mo{\cite[Chapter 5]{Hofs20}} for an extensive discussion of the recent results for the Ising model. Dembo and Montanari \cite{DemMon10a} prove the existence of the pressure per particle, and identify the critical \colrev{inverse temperature}, in the case where the degree distribution has finite variance. Dommers et al.\ \cite{DomGiaHof10} extend this also to power-law random graphs and identify the critical exponents in \cite{DomGiaHof12}. Central limit theorems were obtained in \cite{GiaGibHofPri15}. These results all apply in the general context of \col{the {\em quenched} Ising model on} locally-tree like random graphs, for which there is a quite complete understanding  (unfortunately, the non-classical central limit theorem at the critical point is not proved in great generality).

Another important approach to the study of disordered systems is the {\em annealed} setting, where the pressure per particle is obtained as the logarithm of the averaged 
partition function (contrary to the quenched setting where one averages the logarithm of the partition function). 
\colrev{Essentially the quenched setting} \changed{concerns}{} \colrev{thermodynamic properties when a typical (fixed) realization
of the disorder is considered, whereas in the annealed setting an average over the disorder distribution is taken
so that the spins `see' an average of the possible random graphs. For an extended discussion about 
quenched and annealed states, see for instance \cite{GiaGibHofPri15}.}
Often the annealed setting is easier to handle, and therefore
is used as an approximation to the quenched setting. Examples include spin glasses \cite{ConGia13}, for instance the Sherrington-Kirkpatrick model where  
the annealed partition function becomes trivial due to the independence of the couplings, and random polymers, where the comparison between quenched and annealed is used to distinguish
between relevant and irrelevant disorder \cite{dH}.

For our  setting of the Ising model on a random graph, the fact that the annealed setting is easier than the quenched is not obviously the case in the presence of non-independent edges, 
and also turns out to be wrong when doing the analysis. Essentially, the study of the quenched Ising model on locally-tree like random graphs is reduced to that of the quenched Ising model on random trees, that is solved by a fixed point equation. Here we can crucially rely on the fact that the pressure per particle is a continuous functional in the local-weak topology. The situation is quite different if one considers the {\em annealed} Ising model, since in this case the contribution of graph realizations containing loops can not be ignored.
On the contrary, since the partition function is an expectation of an exponentially large random variable, the contribution of atypical graph realizations (having extremely small probabilities)
to the average of the partition function might be substantial and prevented, so far, the possibility of a general analysis of the annealed setting.

In the physics literature, the folklore is that the annealed Ising model on a random graph is equivalent to
an inhomogeneous mean-field model where the coupling between the spins is proportional to their
degrees, see Bianconi \cite{bianconi} and Dorogovtsev et al. \cite{dorogovtsev2008critical}. However, this analogy often misses an exact proof, and this paper investigates this relation in detail for general configuration models.

From a mathematical perspective, the annealed Ising model could so far only be solved on two specific random graphs models: the generalized random graph 
\cite{DomGiaGibHofPri16}, where edges are independent, and the random regular graph \cite{Can17a, Can17b}, where the degrees are all equal  
(see also \cite{GiaGibHofPri16} where the configuration model with degrees 
that are either 1 or 2 was studied by mapping this model to a one-dimensional Ising model). 

In this paper, we extend the analysis of the annealed Ising model 
to the {\em configuration model}, using ideas from, and extending work of Can \cite{Can17a, Can17b}. The main innovations in our paper are as follows:
\begin{itemize}
\item[$\rhd$] The model includes degree variability, and dependence between edges that is due to the fact that the edges need to realize a prescribed
degree sequence;
\item[$\rhd$] The model turns out to be exactly solvable and therefore we identify the annealed pressure as well as the critical value;
\item[$\rhd$] The exact solution is obtained by a careful analysis of the annealed partition function 
with deterministic degrees using the mapping to a random matching problem derived by the first author
in \cite{Can17a}. In the case of i.i.d.\ degrees, this is complemented by a large deviation analysis for the empirical degree 
distribution under the annealed Ising model;
\item[$\rhd$]
Surprisingly, the annealed critical value equals the quenched one for {\em deterministic} degrees, 
while it is strictly smaller for {\em independent and identically distributed} (i.i.d.) degrees.
\end{itemize}

Our solution shows that, as predicted by physicists, the annealed Ising model 
on the configuration model with prescribed degrees is indeed equivalent to a mean-field model, 
i.e., the magnetization can be computed in terms of a scalar quantity that solves
a mean-field equation, see \eqref{v-def}. However, 
the exact mean-field consistency equation turns out to be substantially different from the
one predicted in \cite{dorogovtsev2008critical}. Also, the prediction of the critical value following from that analysis in general turn out to be wrong.

\subsection{Configuration model}
\label{sec-CM}
Fix an integer $n$ that will be the number of vertices in the random graph \col{and denote by $[n]$ the set $\{1,2,\ldots , n\}$}. Consider a sequence of degrees $\bfdit=(d_i)_{i\in[n]}$. 
%The aim is to construct an undirected (multi)graph with $n$ vertices, where vertex $j$ has degree $d_j$. 
Without loss of generality, we assume that $d_j\geq 1$ for all $j\in [n]$, since when $d_j=0$, vertex $j$ is isolated and can be removed from the graph. 
%One possible random graph model is then to take the uniform measure over such undirected and simple graphs. Here, we call a multigraph {\em simple} when it has no self-loops, and
%no multiple edges between any pair of vertices. However, the set of undirected simple graphs with $n$ vertices where vertex $j$ has degree $d_j$ may be empty. For example, in order for such a graph to exist, we must 
Furthermore, we assume that the total degree
    	\eqn{
    	\ell_n=\sum_{j\in [n]} d_j
    	}
is even. When $\ell_n$ is odd we increase the degree $d_n$ by 1. 
For $n$ large, this will not change the results and we will therefore
ignore this effect.

We wish to construct a graph such that $\bfdit=(d_i)_{i\in[n]}$ are the degrees of the $n$ vertices. 
To construct the multigraph where vertex $j$ has degree $d_j$ for all $j\in [n]$, we have $n$ separate vertices and incident to vertex $j$, we have $d_j$ half-edges. Every half-edge needs to be connected to another half-edge to form an edge, and by forming all edges we build the graph. For this, the half-edges are numbered in an arbitrary order from $1$ to $\ell_n$. We start by randomly connecting the first half-edge with one of the $\ell_n-1$ remaining half-edges, \colrev{chosen uniformly at random}. Once paired, two half-edges form a single edge of the multigraph, and the half-edges are removed from the list of half-edges that need to be paired. Hence, a half-edge can be seen as the left or the right half of an edge. We continue the procedure of randomly choosing and pairing the half-edges until all half-edges are connected, and call the resulting graph the {\it configuration model with degree sequence $\bfdit$}, abbreviated as $\CMnd$.

\subsection{Ising model}
\label{sec-IM}
We start by defining Ising models on finite graphs. Consider a graph sequence $(G_n)_{n \geq 1}$, where $G_n=(V_n,E_n)$, with vertex set $V_n=[n]$ and some random edge set $E_n$. To each vertex $i\in [n]$ we assign an Ising spin $\sigma_i = \pm 1$. A configuration of spins is denoted by \col{$\sigma=\{\sigma_i \colon i\in [n]\}$}. The \col{ferromagnetic} {\em Ising model on $G_n$} is then defined by the \col{Boltzmann-Gibbs} distribution 
	\eqn{
	\label{eq-boltzmann}
	\mu_n(\sigma) 
	= \frac{1}{Z_n(\beta, \col{B})} \exp \Big\{\beta \sum_{(i,j) \in E_n} \sigma_i \sigma_j + B \sum_{i \in [n]} \sigma_i\Big\}.
	}
Here, $\beta \geq 0$ is the inverse temperature and \col{$B$ is the external magnetic field}.
%$\underline{B}=\{B_i\colon i \in [n]\} \in \Rbold^n$ is the vector of external magnetic fields. 
%We will write $B$ instead of $\underline{B}$ for a uniform external field, i.e., $B_i=B$ for all $i\in[n]$. 
The partition function $Z_n(\beta,{B})$ is a 
normalization factor given by
	\eqn{
	\label{eq-partit-function}
	Z_n(\beta,{B})
	= \sum_{\sigma \in \{-1,+1\}^n} \exp \Big\{\beta \sum_{(i,j) \in E_n} \sigma_i \sigma_j + B\sum_{i \in [n]} \sigma_i\Big\}.
	}
We let $\langle \cdot \rangle_{\mu_n}$ denote the expectation with respect to the Ising measure, i.e., for every function $f\colon \{-1,+1\}^n \rightarrow \Rbold$,
	\eqn{
	\langle f\rangle_{\mu_n} =  \sum_{\sigma \in \{-1,+1\}^n} f(\sigma) \mu_n(\sigma).
	}
The main quantity we shall study is the {\em pressure} per particle, which is defined as
	\eqn{
	\label{def-pressure-general}
	\psi_n(\beta, B) = \frac{1}{n} \log Z_n(\beta, B),
	}
in the thermodynamic limit of $n \rightarrow \infty$. It turns out that the pressure characterizes the behavior in the Ising model, and the phase transition and related quantities can be retrieved from it by taking appropriate derivatives with respect to the parameters $B$ and $\beta$. For example,  the {\em magnetization} $M_n(\beta,B)$ is given by
	\eqn{
	\label{def-mag-general}
	M_n(\beta,B) = \Big\langle \frac{1}{n} \sum_{i\in[n]} \sigma_i \Big\rangle_{\mu_n}=\frac{\partial}{\partial B} \psi_n(\beta, B).
	}

When dealing with a random graph, such as $\CMnd$, the above variables are all {\em random}. Then, the above corresponds to the so-called {\em \col{random} quenched} setting. However, we often wish to average out over the randomness of the graph. One way to do this, which is called the {\em averaged quenched setting}, is to simply take the expectation in \eqref{eq-boltzmann}.
In this setting, when computing an expectation with respect to the Ising measure  we deal with a ratio of random variables (since they are functions of the random graph), and both numerator and denominator grow exponentially in the large graph limit. Typically, such problems are difficult, as one still has to deal with the ratio of large random variables. 

A common alternative approach is to take the average over the randomness both in the numerator as well as the denominator, giving rise to the {\em annealed} measure. In the annealed \cla{setting}, the measure \mo{\eqref{eq-boltzmann}} is replaced by
	\eqn{
	\label{eq-boltzmann-annealed}
	\mu_n^{\rm \sss {an,d}}(\sigma) 
	= \frac{1}{\expec[Z_n(\beta, {B})]} \expec\Big[\exp \Big\{\beta \sum_{(i,j) \in E_n} \sigma_i \sigma_j + B \sum_{i \in [n]}  \sigma_i\Big\}\Big].
	}
%Here $\expec$ is the expectation over the randomness due to the random graph.
\colrev{Here $\expec$ is the expectation over the randomness of the  graph.
The superscript ``${\rm d}$'' is added to remember that we are dealing with the 
configuration model with {\em deterministic} degrees. Later (see Section \ref{sec-res-iid})
we shall also consider {\em random} degrees, and then  $\expec$ will denote
the averages over all the sources of randomness of the graph.}
%Here} the superscript ``${\rm d}$'' is added to remember that we are dealing with the 
%configuration model with deterministic degrees. 

The main quantity \changed{of interest}{} 
is the {\em annealed pressure}, that is defined as
\eqn{
	\label{def-pressure-ann}
	\psinand(\beta, B) = \frac{1}{n} \log \expec[Z_n(\beta, B)].
	}

See \cite[Chapter 5]{Hofs20} for an extensive discussion of these different settings, as well as an overview of the results on the Ising model on random graphs.

We are interested in the thermodynamic limit of our models, i.e. in the behavior as $n\to \infty$.  Define $\varphiand(\beta, B):=\lim_{n\to \infty} \psinand(\beta, B)$. For the quenched setting, \mo{the thermodynamic limit of the pressure was proved to exist}, while for the annealed setting, we prove the existence \mo{of $\varphiand(\beta, B)$} in this paper. 
\mo{Being the limit of a sequence of convex functions both in $\beta$ and in $B$, the infinite volume pressure $\varphiand(\beta, B)$ is also convex.}
\mo{Criticality of the model manifests itself in the behavior of the {\em annealed spontaneous magnetization}
	\eqn{
	\label{spon-M-def-qu-ann-d}
	\Msand(\beta):=\lim_{B\searrow 0} \frac{1}{B}[\varphiand(\beta, B) - \varphiand(\beta,0)],
	}
which is well-defined since convex functions always have right derivates.}
%\CG{In the affirmative, then we change everywhere $\Mand(\beta,0^{+})$ with $\Mand(\beta)$. }	
In terms of these quantities, the {\em critical inverse temperature}  is defined as
	\be\label{def_beta_c}
	\betacand:=\inf \{\beta > 0\colon \Msand(\beta)>0\}.
	\ee 
An identical characterization holds in the quenched setting, \mo{now in terms
of the quenched spontaneous magnetization.}
%where all $\liminf$s can be replaced with limits.
 Thus, depending on the setting, we can obtain the \mo{{\em  quenched}} and {\em annealed critical points} denoted by $\betacqud$ and $\betacand$, respectively.  When $0<\beta_{c} <\infty$, we say that the system undergoes a {\em phase transition} at $\beta=\beta_{c}$. \changed{We will refer to $\betacqud$ and $\betacand$ as {\em critical values}, to avoid writing `critical inverse temperature' many times.}{}

It is possible that the critical value $\beta_c$, at which the ferromagnetic phase transition -- if any -- occurs, is different for the quenched and annealed settings (it is not hard to prove that $\beta_c$ is the same in the random and averaged quenched setting). However, the general belief in physics is that, \col{for the ferromagnetic Ising model on random graphs}, the annealed and quenched cases have the same critical exponents, and are thus in the same universality class. In this paper, we study the annealed Ising model on the configuration model. We continue by stating our result, both in the case where the degrees are deterministic, as well as the case where they are i.i.d.\ (see Sections \ref{sec-res-deter} and \ref{sec-res-iid}, respectively).

\subsection{Results for deterministic degrees}
\label{sec-res-deter}
We impose certain \emph{regularity conditions} on the degree sequence $\bfdit$. In order to state these assumptions, we introduce some notation. We denote the degree of a uniformly chosen vertex $\Ver$ in $[n]$ by $D_n=d_{\Ver}$. The random variable $D_n$ has distribution function $F_n$ given by
    \eqn{
    \label{def-Fn-CM}
    F_n(x)=\frac{1}{n} \sum_{j\in [n]} \indic{d_j\leq x},
    }
which is the {\em empirical distribution of the degrees.}
We assume that the vertex degrees satisfy the following \emph{regularity conditions:}

\begin{cond}[Regularity conditions for vertex degrees]
\label{cond-degrees-regcond}
~\\
{\bf (a) Weak convergence of vertex \col{degrees}.}
There exists a distribution function $F$ such that, as $n\rightarrow \infty$,
    \eqn{
    \label{Dn-weak-conv}
    D_n\convd D,
    }
where $D_n$ and $D$ have distribution functions $F_n$ and $F$, respectively, and $\convd$ denotes convergence in distribution.
%\\
%Equivalently, for any $x$, 
%    \eqn{
%    \label{conv-Fn-CM}
%    \lim_{n\rightarrow \infty} F_n(x)=F(x).
%    }
Further, we assume that $F(0)=0$, i.e., $\prob(D\geq 1)=1$.\\
{\bf (b) Convergence of average vertex degrees.} As $n\rightarrow \infty$,
    \eqn{
    \label{conv-mom-Dn}
    \expec[D_n]\rightarrow \expec[D]<\infty,
    }
where $D_n$ and $D$ have distribution functions $F_n$ and $F$ from part (a), respectively.\\
{\bf (c) Convergence of second moment vertex degrees.} As $n\rightarrow \infty$,
    \eqn{
    \label{conv-sec-mom-Dn}
    \expec[D_n^2]\rightarrow \expec[D^2]\in (0,\infty],
    }
where again $D_n$ and $D$ have distribution functions $F_n$ and $F$ from part (a), respectively.\\
{\bf (d) Bound on the maximum degree.} Let $\dmax = \max\{d_i \colon i \in [n]\}$ be the maximum degree
of $n$ vertices. As $n\rightarrow \infty$,
	\eqn{
	\label{max-deg}
	\dmax=o(n/\log{n}).
	}
\end{cond}

While Condition \ref{cond-degrees-regcond}(c) is not necessary for the thermodynamic limits proved in this paper, the second moment of the degrees is needed to identify the critical value of the Ising model on the configuration model, which is why we have added it. Note that we allow for $\expec[D^2]=\infty$, in which case the asymptotic degree distribution has an infinite second moment. Condition \ref{cond-degrees-regcond}(c) rules out cases where $\expec[D_n^2]$ oscillates, in which case we cannot define the critical value. 

As for Condition \ref{cond-degrees-regcond}(d), the bound on the maximal degree is quite weak, in that it will be true for most degree distributions, and is used to characterize the pressure per particle. See (2.10) and (2.14) below.

%\col{We denote by $\psinand(\beta,B)$ the annealed pressure \eqref{def-pressure-ann} with deterministic degree sequence $\bfdit$ satisfying Condition \ref{cond-degrees-regcond}.}
\mo{Two canonical examples of such degree sequences 
that satisfy Condition \ref{cond-degrees-regcond}} 
are when we take $d_i=[1-F]^{-1}(i/n)$ \mo{for \changed{$i\in[n]$}}, where $F$ is the distribution function of an integer-valued random variable, and when $(d_i)_{i\in[n]}$ constitutes a \col{{\em realization} of an} i.i.d.\ sequence of random variables with distribution function $F$. In the latter case, beware that in the annealed measure in \eqref{eq-boltzmann-annealed}, we do {\em not} take the average with respect to the randomness of the degrees. 
\col{The convergence results in this paper are then meant to holds {\em in probability}.}
See \cite[Chapter 7]{Hofs17} for an extensive discussion of the configuration model and the Degree Regularity Condition \ref{cond-degrees-regcond}. 

The \col{Ising model on the configuration model with deterministic degrees} has been investigated in substantial detail in the quenched setting. One of the main results is that the thermodynamic limits of the pressure and magnetization exist, and that the Ising model has a phase transition with the quenched critical value given by 
	\eqn{
	\label{betac-qu}
	\betacqud = \atanh(1/\nu),
	}
where 
	\eqn{
	\label{nu-def}
	\nu=\frac{\expec[D(D-1)]}{\expec[D]}
	}
equals the asymptotic expected forward degree in $\CMnd$. Our main result for the annealed Ising model on the configuration model with deterministic degrees is the following theorem, which settles the same result in the annealed setting:

\begin{theorem}[Thermodynamic limit of the annealed pressure: deterministic degrees]
\label{thm-pressure-CM-ann-deter}
Consider $\CMnd$ where the degree sequence satisfies Conditions \ref{cond-degrees-regcond}(a)-(b) and (d). Then, for all \col{$\beta \ge 0$} and $B\in\Rbold$, the thermodynamic limit 
of the annealed pressure exists, i.e., as $n \to \infty$,
	\eqn{
	\label{pressure-CM-ann-deter}
	\psinand(\beta,B) \rightarrow \varphiand(\beta, B) ,	
	}
where $\varphiand(\beta, B)$ is defined in \eqref{varphiand-def} below.
\end{theorem}
\medskip

We next identify the critical value of the annealed Ising model on the configuration model:

\begin{theorem}[Critical value of annealed Ising model: deterministic degrees]
\label{thm-crit-CM-ann-deter}
Consider $\CMnd$ where the degree distribution satisfies Conditions \ref{cond-degrees-regcond}(a)-(d). Then, the critical value $\betacand$ equals
	\eqn{
	\label{betac-an-d}
	\betacand=\betacqud = \atanh(1/\nu),
	}
where $\nu$ is the expected forward degree in $\CMnd$ defined in \eqref{nu-def}, and $D$ is the asymptotic degree whose existence is stated in  Conditions \ref{cond-degrees-regcond}(a).
%The same result applies under Conditions \ref{cond-degrees-regcond}(a)-(b) and (d), when we instead assume that $\expec[D^2]=\infty$ and we define $\betacand=\beta_c^{\sss\mathrm{qu}} =0$.
\end{theorem}

The above result is quite surprising, as it states that\mo{, when the degrees are fixed,} the configuration model cannot really favorably arrange itself so as to decrease the critical value of the Ising model on it. 
\colrev{Some insight on this result can be obtained by comparing to the corresponding findings for 
the Ising model on the Generalized Random Graph (GRG), where} \changed{instead}{} \colrev{the
annealed critical value is smaller that the quenched critical value \cite{GiaGibHofPri16}. The reason for this  is to be found in the fact that under the annealed measure the GRG random graph re-arranges itself by increasing the typical total number of edges (see \cite{DomGiaGibHof18}). As a consequence, by the second Griffiths inequality, the critical temperature increases. 
A re-arrangement of the degree sequence to increase the number of edges is, of course, 
not possible in the configuration model with prescribed deterministic degrees.}

%\CG{Removed text on power-law random graph since i) it seemed out of place and ii) it was essentially repeated is section 1.5}
%In many recent papers, the Ising model has been considered on random graphs with power-law degrees, as these are supposed to model real-world networks more accurately. 
%This can be achieved by taking an asymptotic degree distribution $p_k=\prob(D=k)$ that satisfies that $\sum_{\ell>k} p_{\ell} =c k^{-(\tau-1)}(1+o(1))$ as $k\rightarrow \infty$ 
%\mo{with $\tau>1$ and} for some $c>0$. In this case, one can expect that the critical exponents in $\CMnd$ where the degree distribution satisfies Conditions \ref{cond-degrees-regcond}(a)-(d) are equal to those in the quenched setting. There, the critical exponents are the ones for the Curie-Weiss model when $\tau>5$ and depend on $\tau$ when $\tau\in(3,5)$. When $\tau\in(2,3)$, we see that $\betacand=\betacqud = \atanh(1/\nu)=0$.

\subsection{Results for i.i.d.\ degrees}
\label{sec-res-iid}
In this section, we continue with the behavior for i.i.d.\ degrees, where now we also average over the randomness in the degrees.
% in \eqref{eq-boltzmann-annealed}.
\mo{We denote the finite-volume annealed pressure 
%\eqref{def-pressure-ann} 
of the annealed Ising model on the configuration models with degrees that are i.i.d.\ copies of the random variable $D$ by} 
\be
\label{pressure-ann-iid}
\psinaniid(\beta,B) =  \frac{1}{n} \log \expec[Z_n(\beta, B)].
\ee
\mo{The superscript ``${\rm D}$'' is added to remember that we are dealing with the 
configuration model with random degrees.} 
\mo{We define $\varphianiid(\beta, B):=\lim_{n\to \infty} \psinaniid(\beta, B)$. 
Provided that this limit is not $\infty$,
the annealed spontaneous magnetization
and the inverse critical temperature
%$\Maniid(\beta,B)$,
$\Msaniid(\beta)$ and $\betacaniid$ are defined in  \mo{a similar} way as in 
%\eqref{M-def-qu-ann-d}, 
\eqref{spon-M-def-qu-ann-d} and \eqref{def_beta_c}, respectively.}
 We remark that the {\em quenched} critical inverse temperature $\betacquiid$ is equal to $\betacand$, where the asymptotic degree distribution $D$ in Condition \ref{cond-degrees-regcond}(a) is the degree distribution.

We emphasize that the expectation $\mathbb{E}$ in \eqref{pressure-ann-iid} includes averaging over two sources of randomness: firstly, for a given realization of the degrees, we average over all possible random graphs that can be formed with that given degrees sequence; secondly, we average over the randomness of the degrees.
It turns out that this has a rather dramatic effect, contrary to the setting of deterministic degrees, that is very much alike
to the quenched setting. Indeed, when the degrees have sufficiently thick tails, the annealed pressure can even be infinite:

\begin{proposition}[Infinite pressure for annealed Ising model with i.i.d.\ degrees]
\label{prop-inf-press-iid}
The pressure of the annealed Ising model on the configuration model with i.i.d.\  degrees that are copies of the random
variable $D$  \mo{with distribution function $F$} equals $\psinaniid(\beta,B)=\infty$ precisely when
	\eqn{
	\label{press-inf-ann-iid}
	\expec[\e^{\beta D/2}]=\infty.
	}
%Further, $\limsup_{n\rightarrow \infty} \psinand(\beta,B)<\infty$ when $\expec[\e^{\beta D/2}]<\infty$.
\end{proposition}

Proposition \ref{prop-inf-press-iid} implies that for power-law degrees, the pressure is simply equal to infinity. The boundary occurs when the degree distribution has an exponential tail, in which case, the pressure per particle is infinity when $\beta$ is large, but not when it is small. In this case, 
%it should be true 
it is natural to believe that the critical value $\betacaniid$ is strictly smaller than the value $\beta$ where the pressure per particle explodes. 
%\RvdH{Do we prove this? I think not. Should we?}
%\CG{I propose we remove the last sentence.}

One could say that the anomaly that $\psinaniid(\beta,B)=\infty$ in Proposition \ref{prop-inf-press-iid} can only occur due to the lack of a single-edge constraint. Thus, it would be interesting to investigate the setting where multiple edges are merged and self-loops removed, which is sometimes called the {\em erased configuration model}. In this case, it is not hard to see that even though $\psinaniid(\beta,B)<\infty$ is always true, it also holds that $\psinaniid(\beta,B)$ tends to infinity with $n$.
% times $n$.
\medskip

From now on, we assume that $\beta$ and $D$ are fixed so that $\expec[\e^{\beta D/2}]<\infty$. Then, the following theorem describes the thermodynamic limit of the pressure per particle for the annealed Ising model on the configuration model with i.i.d.\ degrees:

\begin{theorem}[Thermodynamic limit of the annealed pressure: i.i.d.\ degrees]
\label{thm-pressure-CM-ann-iid}
Consider \col{the configuration model} where the degrees are  i.i.d.\ \mo{copies} of the random variable $D$ with distribution function $F$. 
 \col{Then, for all $\beta$ satisfying $\expec[\e^{\beta D/2}]<\infty$ and for all $B\in\Rbold$}, the thermodynamic limit of the annealed pressure exists, i.e., \col{as $n \to \infty$},
	\eqn{
	\label{pressure-CM-ann-iid}
	\psinaniid(\beta,B) \rightarrow \varphianiid(\beta, B),
	}
where $\varphianiid(\beta, B)$ is defined in \eqref{iid} below.
Further, $ \varphianiid(\beta, B)> \varphiand(\beta, B ; \boldsymbol{p})$, where $\varphiand(\beta, B ; \boldsymbol{p})$ denotes the pressure of
the configuration model with a deterministic degree sequence having  asymptotic degree distribution $\boldsymbol{p} = (p_k)_{k\geq 1}$ with $p_k=\prob(D=k)$.
\end{theorem}
\medskip

Theorem \ref{thm-pressure-CM-ann-iid} implies that it is favorable to increase the edge-density in the \mo{configuration model} with i.i.d.\ degrees, and that the above effect is so pronounced that it changes the pressure per particle. We continue by studying the critical value $\betacaniid$, and show that this effect is so large that it actually changes the critical value:
%
%Before doing so, we first define $\betacaniid$, which is not obvious since we do not know that $B\mapsto \varphianiid(\beta, B)$ is differentiable. Instead, we let
%	\eqn{
%	\label{beta-c-an-iid-def}
%	\betacaniid=\inf\{\beta \colon \liminf_{B\searrow 0} \frac{1}{B}[\varphianiid(\beta, B)-\varphianiid(\beta, 0^+)]>0\}.
%	}
%Since $|[\varphianiid(\beta, B)-\varphianiid(\beta, 0^+)]/B|\leq 1,$ the ratio is bounded, and therefore the $\liminf$ exists. We believe that the $\liminf$ actually is a {\em limit}, and the limit would then be equal to the spontaneous magnetization, but we lack a proof for this. Having defined $\betacaniid$, we now investigate its relation to $\beta_c^{\sss\mathrm{qu}}$:

\begin{theorem}[Critical value of annealed Ising model: \col{i.i.d.} degrees]
\label{thm-crit-CM-ann-iid}
Consider %$\CMnd$ 
the configuration model where the degrees are i.i.d.\ copies of the random variable $D$ with distribution function $F$. The critical inverse temperature $\betacaniid$ satisfies $\betacaniid\leq \bar{\beta}_c^{\sss \rm D}$, where $\bar{\beta}_c^{\sss \rm D}$ is the unique solution to
%\CG{change $\bar{\beta}_c^{\sss \rm D}$ to $\bar{\beta}^{\sss \rm D}$ everywhere? }
	\eqn{
	\label{bcand}	
	\bar{\beta}_c^{\sss \rm D}=\atanh(1/\nu(\boldsymbol{q}(\bar{\beta}_c^{\sss \rm D}))),
	}
where 	
	\eqn{
	\label{nu-q-beta-def}
	\nu(\boldsymbol{q})=\frac{\sum_{k\geq 1}k(k-1)q_k}{\sum_{k\geq 1} kq_k}
	\qquad
	\text{and}
	\qquad q_k(\beta)=\frac{p_k \cosh(\beta)^{k/2}}{\expec[\cosh(\beta)^{D/2}]},\quad \cla{\text{for }k\in \N}
	}
and $p_k=\prob{(D=k)}$.
Assume that $\expec[\e^{\bar{\beta}_c^{\sss \rm D} D/2}]<\infty$. Then, the critical value $\betacaniid$ satisfies
	\eqn{
	\label{betac-an-iid-bd}
	\betacaniid<\betacqud= \atanh(1/\nu),
	}
unless $D=r$ almost surely for some \mo{positive integer} $r$.
\end{theorem}

%\RvdH{Think about what happens when $\expec[\e^{\bar{\beta}_c^{\sss \rm D} D/2}]=\infty$, but $\expec[\e^{\betacaniid D/2}]<\infty$.}
%\medskip

The case where $D=r$ corresponds to random regular graphs, and is the only case where Theorems \ref{thm-pressure-CM-ann-deter} and \ref{thm-pressure-CM-ann-iid} actually agree, which explains why this case needs to be excluded in Theorem \col{\ref{thm-crit-CM-ann-iid}}. The random regular graph case was also investigated by Dommers et al.\ in \cite{GiaGibHofPri16} for $r=2$, by Can in \cite{Can17a, Can17b}, \col{and by Dembo, Montanari, Sly and Sun \cite{DemMonSlySun14} (see also the remark below \cite[Theorem 1]{DemMonSlySun14}, where it is mentioned that the Ising result also holds for odd degree).}
%\CG{what do we mean by Ising result?}

The fact that $\betacaniid<\betacqud$ can be understood by noting that, by a correlation inequality,
	\eqn{
	\nu(\boldsymbol{q}(\beta))\geq \nu.
	}
\col{Indeed, introducing the random variable  $D^\star$ with law given by 
	\eqn{\label{Dstar}
	\pp(D^\star=k) = \frac{k p_k}{\expec[D]},
	}	
where $p_k = \pp(D=k)$, one has}	
	\eqn{
	\nu(\boldsymbol{q}(\beta))=\frac{\expec[(D^\star-1)\cosh(\beta)^{D^\star/2}]}{\expec[\cosh(\beta)^{D^\star/2}]}\geq \expec[D^\star-1]=\nu,
	}
with strict inequality unless $D^\star$ is constant or $\beta=0$. Thus, \mo{by the first part of Theorem \ref{thm-crit-CM-ann-iid}}, 
$$
\mo{\betacaniid \le \bar{\beta}_c^{\sss \rm D} =\atanh(1/\nu(\boldsymbol{q}(\bar{\beta}_c^{\sss \rm D})))<  \atanh(1/\nu) = \betacqud.}
$$ 

\begin{remark}[Equality of $\betacaniid$ and $\bar{\beta}_c^{\sss \rm D}$]
\label{rem-eq-crit-betas-iid}
{\rm We believe that $\betacaniid$ and $\bar{\beta}_c^{\sss \rm D}$ are equal, see also Remark \ref{rem-eq-crit-betas-iid-rep} below, where we explain this intuition more precisely.
Unfortunately, due to insufficient analytical control of the variational formulas that define $\varphianiid(\beta, B)$, we are unable to prove this.}
\end{remark}

\begin{remark}[Specific degree distributions]
{\rm In Section \ref{sec-crit-value-iid}, we work out two special cases. For the CM with Poisson degrees with parameter $\lambda$, we have $\beta_c^{\sss\mathrm{qu}}=\atanh(1/\lambda)$ for $\lambda>1$. When $\lambda\leq 1$, there is no giant component \cla{\cite{Hofs18}}, so that $\beta_c^{\sss\mathrm{qu}}=\infty$. In the annealed setting, however, we instead obtain that
	\eqn{
	\label{pippo}
	\bar{\beta}_c^{\sss \rm D}= -\log (2 \lambda^2) + \log   \left[1+\sqrt{1+4 \lambda^4}+\sqrt{2 +2 \sqrt{1+4 \lambda^4}}\right]<\infty,
	}
for all $\lambda>0$, so that the annealed setting with \mo{Poisson} i.i.d.\ degrees behaves markedly different from the quenched setting. In particular, even when $\lambda<1$, the annealed Ising model still has a {\em finite} critical value, indicating that the graph, under the annealed Ising model, will have a giant component for {\em all} $\lambda>0$. 
%Note further that the critical value is different from the one for the annealed Erd\H{o}s-R\'{e}nyi random graph, which is \mo{$\arcsinh(1/\lambda)$} \cite{DomGiaGibHofPri16}, even though the (quenched) asymptotic degree distributions are equal.\\
%
%\RvdH{Is this still true, since now we have an inequality?}
%\CG{Can we find a range of values of $\lambda$ so that \eqref{pippo} < $\arcsinh(1/\lambda)$?}
%\Hao{Unfortunately, it is not true. In fact, we can prove that $ \bar{\beta}_c^D > \asinh(1/\lambda)$ for all $\lambda$.} \mo{Indeed, $\asinh(1/\lambda)=\log[1+\sqrt{1+\lambda^2}] -\log (\lambda)$ and 
%	$$ \bar{\beta}_c^D > -\log (2\lambda^2) + \log \left[1+\sqrt{1+4\lambda^4}+ 2\lambda\right].$$ 
%Thus
%$$\bar{\beta}_c^D-\asinh(1/\lambda) > \log\left[1+\sqrt{1+4\lambda^4}+ 2\lambda\right] - \log \left[2\lambda(1+\sqrt{1+\lambda^2})\right]. $$
%Observe that 
%\begin{eqnarray*}
%1+\sqrt{1+4\lambda^4}+ 2\lambda > 2\lambda(1+\sqrt{1+\lambda^2})  \quad \Leftrightarrow \quad 2+4 \lambda^4+ 2\sqrt{1+4\lambda^4} >4 \lambda^4 + 4\lambda^2,
%\end{eqnarray*}
%which holds for all $\lambda >0$. Thus $\bar{\beta}_c^D>\asinh(1/\lambda)$ for all $\lambda>0$.
%}	

\cla{A similar scenario appears for the CM with Geometric degrees with parameter $p$. In this case  $\beta_c^{\sss\mathrm{qu}}= \atanh(p/[2(1-p)])$, since the expected forward degree is $\nu=2(1-p)/p$. Thus, when $p\ge \frac 2 3$ there is no giant component and $\beta_c^{\sss\mathrm{qu}}=\infty$. On the other hand, in the annealed case with i.i.d.\ degrees we prove that
	\eqn{
	\bar{\beta}_c^{\sss \rm D}= \ln \left ( {x^\star(p)}^2 + \sqrt{{x^\star(p)}^4-1}\right) < \infty,
	}
where $x^\star(p)$ is the solution of a fourth order algebraic equation, see \eqref{eqn-geom-eqn}, that is larger than 1 for {\em all} $0<p<1$. Also in this case we have a {\em finite} annealed critical value for all values of the parameter $p$. 
}
}
\end{remark}

\subsection{Discussion}
\label{sec-disc}
In this section, we provide some discussion of our results and state open problems.

%\medskip
%\paragraph{\bf Deterministic degrees: equality of \cla{pressures}.} For the random $r$-regular graph, not only do the critical values in the annealed and quenched settings agree, but even their \cla{pressures \cite{Can17a}}. We believe that the quenched and annealed \cla{pressures} are equal more generally, \cla{i.e. for all configuration models with deterministic degrees,} and this would give a much more convenient analytical description of it, compared to the results in \cite{DemMon10a, DomGiaHof10}. However, it is not clear to us how to prove such a result.

\medskip
\paragraph{\bf Quenched and annealed.} \mo{For the random $r$-regular graph, not only do the critical values in the annealed and quenched settings agree, but even their {pressures \cite{Can17a}}. 
We believe however this to be exceptional.  Generically, i.e. for all configuration models with non constant deterministic degrees, we proved in Theorem \ref{thm-crit-CM-ann-deter} that the critical 
values are the same, but we believe that the quenched pressure in \cite{DemMon10a, DomGiaHof10} and the annealed {pressures} in Theorem \ref{thm-pressure-CM-ann-deter} are different, 
as the first involves the solution of an infinite dimensional variational problem for the so-called cavity field
distribution, whereas the second involves the solution of a scalar variational problem (see Lemma \ref{lem-sk-sol}).  However, it is not clear to us how to prove such a result.
}

\medskip
\paragraph{\bf Deterministic degrees: several universality classes.} For deterministic degrees, we expect that \mo{in the annealed setting} the behavior \mo{around criticality} is described by the degree distribution in a similar way as for the quenched case, as described in \cite{DomGiaHof12}. In particular, such a result would imply that in the case where the degrees have a bounded fourth moment, the critical exponents are equal to those for the Curie-Weiss model, while in the strongly inhomogeneous setting of power-law degrees with exponent $\tau\in(3,5)$, they depend on $\tau$. 
\mo{We do not address this question in this paper.}
%This would follow directly from the equality of partition functions as conjectured in the previous comment. 
%However, it might also be proved using direct computations. 

\medskip
\paragraph{\bf Independent and identically distributed degrees: only one universality class.} The situation is quite different in the case of i.i.d.\ degrees. 
{\rm We believe that $\betacaniid=\bar{\beta}_c^{\sss \rm D}$ (recall Remark \ref{rem-eq-crit-betas-iid}). When $\expec[\e^{\betacaniid D/2}]<\infty$, then $\boldsymbol{q}(\betacaniid)$ in \eqref{nu-q-beta-def} has exponential tails, since $\cosh(\betacaniid)<\e^{\betacaniid}$. Therefore, denoting $\boldsymbol{q}(\beta,B)$ to be the empirical degree distribution of the annealed Ising model on the configuration model with i.i.d.\ degrees, one can expect that for $\beta>\betacaniid$ and $B>0$ with $\beta-\betacaniid$ and $B$ very small, also $\boldsymbol{q}(\beta,B)$ has exponential tails. As a result, power-law degree distributions cannot occur, which suggests that the critical exponents are all equal to those of the Curie-Weiss model. In this case, there exists only {\em one} universality class, compared to the several ones for the setting of deterministic degrees. It would be quite surprising should the universality classes for the two settings be so crucially different. See Remark \ref{rem-crit-iid} for more details.
\medskip

\paragraph{\col{\bf Generalized random graph with random weights.}} 
\col{
%As it will be discussed in Section \ref{sec-grg} 
The analysis developed for the configuration model with i.i.d.\ degrees can be extended to 
the generalized random graph model with i.i.d.\ weights. For generalized random graphs with deterministic \mo{weights}, it is  known that the value of
the critical temperature where the annealed phase transition occurs is  strictly larger than the value of the 
quenched critical temperature \cite{GiaGibHofPri16}. The additional randomness of the weights causes a further increase of the
annealed critical temperature. As for the critical behavior, by the same argument of the previous item we expect the existence
of a unique universality class for the generalized random graph with i.i.d. weights, contrary to the quenched case or the annealed case with deterministic weights
that have been shown to have multiple universality classes \cite{DomGiaHof12, DomGiaGibHofPri16}.
Thus adding randomness to the random graph seems to favor a higher degree of universality, in the sense
that different universality classes merge together as a consequence of averaging.}
\medskip

\paragraph{\bf Organization of the paper.} We prove the results for deterministic degrees in Section \ref{sec-proof-deter}. More precisely, we identify the partition function for deterministic degrees and prove Theorem \ref{thm-pressure-CM-ann-deter} in Section \ref{sec-part-func-deter}, and identify the annealed critical value in Theorem \ref{thm-crit-CM-ann-deter} in Section \ref{sec-crit-value-deter}. We prove the results for i.i.d.\ degrees in Section \ref{sec-proof-iid}. More precisely, we show when the partition function is finite and prove Proposition \ref{prop-inf-press-iid} in Section \ref{sec-proof-prop-inf-press-iid}, identify the partition function for i.i.d.\ degrees with finite support and prove Theorem \ref{thm-pressure-CM-ann-deter} in Section \ref{sec-part-func-iid-fin}, extend the analysis to infinite support i.i.d.\ degrees in Section \ref{sec-part-func-iid-inf}, and identify the annealed critical value in Theorem \ref{thm-crit-CM-ann-iid} in Section \ref{sec-crit-value-iid}.
%Finally, Section \ref{sec-grg} contains the extension to the generalized random graph with i.i.d.\ random weights.

\section{Deterministic degrees: Proof of Theorems \ref{thm-pressure-CM-ann-deter} and \ref{thm-crit-CM-ann-deter}}
\label{sec-proof-deter}
In this section, we prove our results in Theorems \ref{thm-pressure-CM-ann-deter} and \ref{thm-crit-CM-ann-deter} 
on the annealed Ising model on the configuration model with deterministic degrees satisfying Conditions \ref{cond-degrees-regcond}(a)-(b) and (d).
In Section \ref{sec-part-func-deter}, we derive the thermodynamic limit of the pressure per particle. In Section \ref{sec-crit-value-deter}, we identify the critical value.

\subsection{Partition function for deterministic degrees: Proof of Theorem \ref{thm-pressure-CM-ann-deter}}
\label{sec-part-func-deter}
Consider the annealed Ising model on the configuration model $\CMnd$. Assume that Conditions \ref{cond-degrees-regcond}(a)-(b) and (d) hold. For $k\geq 1$, denote
	\eqn{
	n_k= \#\{i \leq n\colon d_i  =k\}, \hspace{2 cm} \textrm{and} \hspace{2 cm} \pkn= \prob(D_n=k)=\frac{n_k}{n}.
	}
By Condition \ref{cond-degrees-regcond}(a), $\pkn\rightarrow p_k=\prob(D=k)$. By symmetry, we only consider the case that $B \geq 0.$

We now rewrite the partition function, using ideas from Can  \cite{Can17a}. By \cite[(3.4)]{Can17a}, the expectation of the partition function can be written as 
	\begin{equation} \label{ezn}
	\E[Z_n(\beta,B)]= \e^{(\beta \expec[D_n]/2 - B) n} \sum \limits_{A \col{\subseteq} [n]} \e^{2B |A|} g_{\beta}(\ell_A, \ell_n),
	\end{equation}
where $|A|$ denotes the number of vertices in $A$, 
	\eqn{
	\ell_A = \sum\limits_{i \in A} d_i
	}
is the total degree of the vertices in $A$ and
	\eqn{
	g_{\beta}(k,m)=\expec[\e^{-2\beta X(k,m)}].
	} 
Here, for $k\in[m]$, $X(k,m)$ denotes the number of edges between $[k]$ and $[m]\setminus [k]$ in a configuration model with $m$ vertices of degree 1 only (i.e., a random matching).

The proofs of Can  \cite{Can17a, Can17b} are centered around a careful asymptotic analysis of the function $g_\beta(k,m)$ that we next explain in detail (see in particular \cite[Lemma 2.1]{Can17a}). Can applies these ideas to the random regular graph, but also notices that the results hold more generally (see \cite[Section 7]{Can17a}). In this paper, we extend the analysis so as to deal with the non-regular case, for which Can proves the existence of the annealed pressure in \cite[Proposition 7.3]{Can17a}, but does not analyze it further. Indeed, the sequence $g_{\beta}(k,m)$ satisfies
 	\eqn{
	\label{gbeta-asy}
	\max_{0 \leq k \leq m} \,\,\vline\frac{\log g_\beta(k,m)}{m}- F_{\beta}(k/m)  \, \vline \, = \frac{\kO(1)}{m},
  	}
where, denoting $x\wedge y=\min\{x,y\}$ for $x,y\in \R$,
  	\eqn{
	\label{Fbeta-def}
	F_{\beta}(t)= \int_0^{t \wedge (1-t)} \log f_{\beta}(u)du,
	}
and
	\eqn{
	\label{fbeta-def}
	f_{\beta}(u) = \frac{\e^{-2\beta}(1-2u)+ \sqrt{1+(\e^{-4\beta}-1)(1-2u)^2}}{2(1-u)}.
	}

Using these asymptotics, \col{a variational formula for the pressure per particle can be obtained in the thermodynamic limit. We start by rewriting} \eqref{ezn} as 
	\begin{equation} \label{nezn}
	\E[Z_n(\beta,B)]= \e^{(\beta \expec[D_n]/2 - B) n} \sum \limits_{(j_k) \leq (n_k)} \prod\limits_{k \geq 1} \binom{n_k}{j_k} \e^{2B \sum_{k\geq 1} j_k} g_{\beta} \left(\sum_{k\geq 1} kj_k, \ell_n\right).
	\end{equation}
In the above formula, compared to \eqref{ezn}, $j_k$ is the number of vertices of degree $k$ that are in $A$ (so that indeed $j_k\leq n_k$), and $\ell_A=\sum_{k\geq 1} kj_k$. The factor $\binom{n_k}{j_k}$ arises from the number of possible choices of $j_k$ elements out of a total of $n_k$ elements.

Then 
%	\eqan{
%	\label{limiting-VP}
%	\psinand(\beta,B)&= \frac{1}{n} \log \E[Z_n(\beta,B)]\nn\\
%	&=\frac{\beta \expec[D_n]}{2} - B + \sum\limits_{(j_k) \leq (n_k)} \left[ \frac{1}{n} \sum_{k\geq 1} \log \binom{n_k}{j_k} + 2 B \sum_{k \geq 1} \frac{j_k}{n} 
%	+ \frac{1}{n} \log g_{\beta} \left(\sum_{k \geq 1} kj_k, \ell_n\right) \right].\nn
%	}
\col{		\eqan{
	\label{limiting-VP}
	\psinand(\beta,B)=&\frac{\beta \expec[D_n]}{2} - B \nn\\
	&+ \frac{1}{n} \log \sum\limits_{(j_k) \leq (n_k)} \exp\left[\sum_{k\geq 1} \log \binom{n_k}{j_k} + 2 B \sum_{k \geq 1} j_k 
	+ \log g_{\beta} \left(\sum_{k \geq 1} kj_k, \ell_n\right) \right].\nn
	}
}
We can bound the above sum from below by taking the maximum over all $(j_k) \leq (n_k)$, and from above by the total number of $(j_k)$ with $(j_k) \leq (n_k)$ times the maximum.
Note that this number is bounded from above by
	\eqn{
	\prod_{k\geq 1} \Big(1+(n_k-1)\indic{n_k\geq 1}\Big).
	}
Since $\dmax=o(n/\log{n})$ by Condition \ref{cond-degrees-regcond}(d), and since $n_k\leq n$, we can trivially upper bound
	\eqn{
	\prod_{k\geq 1} \Big(1+(n_k-1)\indic{n_k\geq 1}\Big)\leq n^{\dmax}=\e^{\dmax \log{n}}=\e^{o(n)}.
	}
Further, let us denote the proportion of vertices of degree $k$ in our set as
	\eqn{
	s_k = j_k/n_k \in [0,1].
	}
Then, the Stirling approximation $n!=\sqrt{2\pi n} \e^{-n} n^n (1+O(1/n))$ gives the estimate
	\eqn{
	\binom{n}{j}= \frac{n^n}{j^j (n-j)^{n-j}} \e^{O(\log{n})}=\e^{n I(j/n)+O(\log{n})},
	}
where $I(0)=I(1)=0$ and, for $t \in (0,1)$,
	\eqn{
	\label{I-def}
	I(t)=-t \log t - (1-t) \log (1-t).
	}
Using that $\sum_{k\geq 1} \log{n_k}\indic{n_k\geq 1}=o(n)$ when $\dmax=o(n/\log{n})$, Condition \ref{cond-degrees-regcond}(d) implies
	\eqan{
	\label{psinand-1}
	\psinand(\beta,B)&=\frac{\beta \expec[D_n]}{2}
	+\max\limits_{0\le s_k \leq 1}\Big[  \sum_{k\geq 1} \pkn I(s_k) + B \Big(2\sum_{k \geq 1} s_k\pkn -1\Big)\nn\\
	&\qquad\qquad\qquad\qquad\qquad + \expec[D_n] F_{\beta}\Big(\frac{\sum_{k\geq 1}k\pkn s_k}{ \expec[D_n] }\Big) \Big] +O\Big(\frac{\dmax \log{n}}{n}\Big),
	}
where the maximum runs over $s_k=j_k/n_k$ and we have used \eqref{gbeta-asy}. We next argue that we can replace the maximum over $s_k=j_k/n_k$ by the supremum over $s_k\in(0,1)$. 
The upper bound holds trivially; for the lower bound, we will make use of the explicit solution that will be derived below, which satisfies that $s_k^\star \in (\vep,1-\vep)$ for all $k\geq 1$. See Lemma \ref{lem-sk-sol} below. Since the function $\sum_{k\geq 1} \pkn I(s_k) + B \Big(2\sum_{k \geq 1} s_k\pkn -1\Big)+ \expec[D_n] F_{\beta}\Big(\frac{\sum_{k\geq 1}k\pkn s_k}{ \expec[D_n] }\Big)$ is uniformly continuous for $s_k\in (\vep,1-\vep)$, this shows that the maximum will be close to the supremum over $s_k\in(0,1)$.

Using that $\pkn\rightarrow p_k$ by Condition \ref{cond-degrees-regcond}(a) and $\expec[D_n]\rightarrow \expec[D]$ by Condition \ref{cond-degrees-regcond}(b), we see that the function that is being maximized converges pointwise to
	\eqn{
	\label{G(sk)-def}
	G((s_k)_{k\geq 1})=\sum_{k\geq 1} p_k I(s_k) + B \Big(2\sum_{k \geq 1} s_kp_k -1\Big)+ \expec[D] F_{\beta}\Big(\frac{\sum_{k\geq 1} k p_k s_k}{ \expec[D] }\Big).
	}
We thus arrive at
	\eqan{
	\label{limiting-VP}
	\psinand(\beta,B)&=\frac{\beta \expec[D]}{2}+\sup\limits_{s_k\in(0,1)}G((s_k)_{k\geq 1})+o(1).
	}
This proves Theorem \ref{thm-pressure-CM-ann-deter}, with
	\eqn{
	\label{varphiand-def}
	\varphiand(\beta,B)=\frac{\beta \expec[D]}{2}+\sup\limits_{s_k\in(0,1)}G((s_k)_{k\geq 1}).
	}
We first show that the maximizers of $G$ are attainable. To do this, we will show $G$ is a continuous function in an appropriate compact metric space. 
Let us define 
    $$\kA=[0,1]^{\mathbb{N}}=\{\boldsymbol{s} = (s_k)_{k\geq 1} \colon 0 \leq s_k \leq 1\}.$$
We define a metric on $\kA$ associated with the degree distribution $(p_k)_{k\geq 1}$ as follows. For $\boldsymbol{s}=(s_k)_{k\geq 1}$ and $\boldsymbol{t}=(t_k)_{k\geq 1}$, define \col{the distance}
   $$d(\boldsymbol{s}, \boldsymbol{t}) = \sum_{k\geq 1} kp_k|s_k-t_k|+\sum_{k\geq 1} 2^{-k}|s_k-t_k|\indic{p_k=0}.$$ 
Since $\sum_{k\geq 1} kp_k$ is finite, by standard arguments in functional analysis, we can show that  $(\kA,d)$ is a compact metric space. 
\cla{Indeed, $d(\boldsymbol{s},\boldsymbol{t})$ metrizes the product topology  \cite{Dug66}  according to which  $[0,1]^{\mathbb{N}}$ is compact by Tychonoff's theorem.}
We now show that $G$ is continuous function on $(\kA,d)$.  Suppose that $\boldsymbol{s}^{\sss(n)}$  converges to  $\boldsymbol{s}$ in $(\kA,d)$. We shall show  $G(\boldsymbol{s}^{\sss(n)}) \rightarrow G(\boldsymbol{s})$ as $n\rightarrow \infty$.  By the  definition of $G$ and $d$, 
    \eqan{
    \label{gsns}	
    |G(\boldsymbol{s}^{\sss(n)})-G(\boldsymbol{s})| &\leq \sum_{k\geq 1} p_k |I(s^{\sss(n)}_k)-I(s_k)|+2B \sum_{k\geq 1}p_k|s^{\sss(n)}_k-s_k| \nn\\
     & \qquad + \expec[D] \Big | F_{\beta}\Big(\frac{\sum_{k\geq 1} k p_k s^{\sss(n)}_k}{ \expec[D] }\Big)-  F_{\beta}\Big(\frac{\sum_{k\geq 1} k p_k s_k}{ \expec[D] }\Big) \Big | \nn\\
     & \leq \sum_{k\geq 1} p_k |I(s^{\sss(n)}_k)-I(s_k)| + \Big[2B+\sup_{t\in (0,1)}|F_{\beta}'(t)| \Big] d(\boldsymbol{s}^{\sss(n)},\boldsymbol{s}).
    }
We notice that 
    \eqn{
    	\label{supFp}
    \sup_{t\in (0,1)}|F_{\beta}'(t)| = 2 \beta,
    }
 since 
  \eqn{ \label{fbp}
  	F_{\beta}'(t) = \begin{cases} \log f_{\beta}(t) & \textrm{ if } t\leq 1/2, \\
  		- \log f_{\beta}(1-t) & \textrm{ if } t\geq 1/2,
  	\end{cases}
  }
and $\e^{-2 \beta} \leq f_{\beta} (t) \leq 1$ for all $t\in (0,\tfrac{1}{2})$. Therefore, the second term in \eqref{gsns} converges to $0$ as $n\rightarrow \infty$. To show the convergence of the first term, we notice that for all $\ell\geq 1$, 
   \eqn{
   	\label{isns}
   \sum_{k\geq 1} p_k |I(s^{\sss(n)}_k)-I(s_k)| \leq \sum_{k\leq \ell} p_k |I(s^{\sss(n)}_k)-I(s_k)| + 2 \log 2 \sum_{k\geq \ell}p_k,
   }
since $0\leq I(s) \leq \log 2$ for all $s\in [0,1]$. The first term converges to $0$ as $n \rightarrow \infty$, since $I(\cdot)$ is a  continuous function and $s^{\sss(n)}_k \rightarrow s_k$ for all $k\leq \ell$; the second term converges to $0$ as $\ell \rightarrow \infty$. Hence, we conclude that the first term in \eqref{gsns} converges to $0$ as $n\rightarrow \infty$, and thus $G$ is a continuous function on a compact space. Therefore, $G$ attains maximizers at some points $\boldsymbol{s} \in \kA$. 

In the next lemma, we identify the solution of this optimization problem, and thus give a more explicit formula for $\varphiand(\beta,B)$. This in particular implies that $\varphiand(\beta,B)$ is well defined and thus completes the proof of Theorem \ref{thm-pressure-CM-ann-deter}. We remark that the functional to be optimized in \eqref{G(sk)-def} only depends
on those components $s_k$ whose index $k\in\mathbb{N}$ is such that the asymptotic degree distribution $p_k = \mathbb{P}(D=k)$ satisfies $p_k> 0$.

\begin{lemma}[Solution of the optimization]
\label{lem-sk-sol}
Let $r=\dmin=\min\{i\geq 1\colon p_i>0\}$ be the minimal asymptotic degree and let ${{\mathcal S}}_r :=\{ i \ge r \,:\, p_i >0\}$. Let $B\geq 0$. 
The maximizer $(s^\star_k)_{k\geq 1}$ to \col{\eqref{varphiand-def}} satisfies,
for all $k \in {\mathcal S}_r$,
	\eqn{
	\label{sk-def}
	s_k^\star= \frac{1}{w^{k} \e^{-2B} +1},
	}
where, for $B>0$, $w=w(\beta,B)$ is a solution in $(\e^{-2\beta},1)$ to
	\eqn{
	\label{v-def}
	\frac{1 - \e^{-2 \beta}w }{1+w^2-2\e^{-2 \beta} w } =  \E  \left[ \left(1+w^{D^\star} \e^{-2B} \right)^{-1}\right],
    }
where $D^\star$ is the random variable defined in \eqref{Dstar}.
\end{lemma}

\proof We first show that if $(s_k^\star)_{k\geq 1}$ is a maximizer of $G$, then for all \col{$k\in {\mathcal S}_r$},
  \begin{itemize}
  	\item[(a)] $1/2 \leq s_k^\star <1$,
  	\item[(b)] $\partial G ((s_k^\star)_{k\geq 1})/\partial s_k=0$.
  \end{itemize}

We prove (a) by contradiction. Assume that there is an index $\ell \in {\mathcal S}_r$ such that $s^\star_{\ell} < \tfrac{1}{2}$. To show that $(s_k^\star)_{k\geq 1}$ is not a maximizer,  it suffices to prove that 
$G(\col{(u_k^{\star})_{k\geq 1}}) > G(\col{(s_k^{\star})_{k\geq 1}})$, with $u_k^\star=s_k^\star \vee (1-s_k^\star)$ for ${k\geq 1}$, \mo{where $ x \vee y= \max \{ x,y\}$}. We  observe that  $u_k^\star\geq \tfrac{1}{2}$ for all $k$, so $\sum_{k\geq 1} k p_k  u_k^\star/\E[D]\geq \tfrac{1}{2}$. Recall that $F_{\beta}(t)=F_{\beta}(1-t)$ for $t\in [0,1]$,
%    \eqn{
%    F_{\beta} \Big(\sum_{k\geq 1} k p_k  s_k^\star/\E[D]\Big) =F_{\beta} \Big(\sum_{k\geq 1} k p_k  (1-s_k^\star)/\E[D]\Big).
%    } 
%Further, $I(s_k^\star)=I(1-s_k^\star)$ for all $k\geq 1$, so that
%	\eqn{
%	\sum_{k\geq 1}  p_k (2s_k^\star  -1)=-\sum_{k\geq 1}  p_k (2(1-s_k^\star)-1),
%	}
%so we may assume that $\sum_{k\geq 1}  p_k (2s_k^\star  -1)\geq 0$ when $B\geq 0$. This means that $s_k^\star\geq \tfrac{1}{2}$ for some $k\geq r$, and $s_k^\star>\tfrac{1}{2}$ for some $k$ when $s^\star_{\ell} < \tfrac{1}{2}$.
%
%\RvdH{Complete this argument!}
%
%and  
%	\eqan{
%		F_{\beta} \Big(\sum_{k\geq 1} k p_k  u_k^\star/\E[D]\Big)&
%		= F_{\beta} \Big(\sum_{k\geq 1} k p_k (s_k^\star \vee (1-s_k^\star))/\E[D]\Big).
%		}
\color{black}and $F_{\beta}$ is decreasing in $(0, \tfrac{1}{2}]$ and increasing in $[\tfrac{1}{2},1)$, while
	\eqan{
	\Big|\frac{1}{\E[D]}\sum_{k\geq 1} k p_k  u_k^\star-\tfrac{1}{2}\Big|
	&=\Big|\frac{1}{\E[D]}\sum_{k\geq 1} k p_k  (s_k^\star \vee (1-s_k^\star)-\tfrac{1}{2})\Big|\\
	&=\frac{1}{\E[D]}\sum_{k\geq 1} k p_k  |\mo{s_k^\star}-\tfrac{1}{2}|\geq \Big|\frac{1}{\E[D]}\sum_{k\geq 1} k p_k  (\mo{s^\star_k}-\tfrac{1}{2})\Big|,\nn
	}
so that $\sum_{k\geq 1} k p_k  u_k^\star/\E[D]$ is further from $\tfrac{1}{2}$ than $\sum_{k\geq 1} k p_k  s_k^\star/\E[D]$. Hence, 
     \eqn{
     	 F_{\beta} \Big(\sum_{k\geq 1} k p_k  u_k^\star/\E[D]\Big) \geq 
     	 F_{\beta} \Big(\sum_{k\geq 1} k p_k  s_k^\star/\E[D]\Big).
          }
\color{black}
Furthermore,
		  \eqn{
		  B \Big(2\sum_{k\geq 1} p_k  u_k^\star  -1\Big) > B\Big(2\sum_{k\geq 1}  p_k s_k^\star  -1\Big).
		  }
Finally, $I(u_k^\star)=I(s_k^\star)$ for all $k\geq 1$. We conclude that $G ((u_k^\star)_{k\geq 1}) > G((s_k^\star)_{k\geq 1})$, which is a contradiction with $(s_k^\star)_{k\geq 1}$ being a maximizer.

Now assume that there is an index $\col{\ell \in\mathcal{S}_r}$, such that $s_{\ell}^\star =1$. We compute 
       \eqn{
       \label{dogk}
       \frac{\partial G}{\partial s_k} = \col{p_k \left(I'(s_k) +2  B + k F_{\beta}' \left(  \frac{\sum_{j\geq 1} j p_j s_j}{\expec[D]} \right)\right)}.
       }
Therefore, $\partial G ((s_k^\star)_{k\geq 1})/\partial s_{\ell}=-\infty$ for $s_{\ell}=1$, since $I'(1)=-\infty$ and $F_{\beta}'$ is a bounded function (see \eqref{supFp}). Hence, there exists a small positive constant $\varepsilon$, such that $G((u_k^\star)_{k\geq 1}) > G((s_k^\star)_{k\geq 1})$, with   $u_k^\star=s_k^\star$ for $k\neq \ell$ and $u_{\ell}^\star=1-\varepsilon$. This completes the proof of (a).

The proof of (b) is similar. Assume that there exists a maximizer $(s_k^\star)_{k\geq 1}$ satisfying $\tfrac{\partial G ((s_k^\star)_{k\geq 1})}{\partial s_{\ell}} \neq 0$ for some $\col{\ell \in\mathcal{S}_r}$. Thanks to part (a), $1/2 \leq s_{\ell}^\star <1$. Hence, there exists $(u_k^\star)_{k\geq 1}$, with $u_k^\star =s_k^\star$ for $k\neq \ell$ and $u^\star_{\ell} \in (s^\star_{\ell}-\varepsilon, s^\star_{\ell}+\varepsilon)$ with $\varepsilon$ small enough, such that $G((u_k^\star)_{k\geq 1}) > G((s_k^\star)_{k\geq 1})$, which again gives a contradiction. Thus, also part (b) follows.

By part (b), all maximizers are solutions to the systems of equations $\partial G/\partial s_k =0$ \col{for all $k \in {\mathcal S}_r$}. By part (a), all maximizers $(s_k)_{k\geq 1}$ satisfy $s_k\geq 1/2$  \col{for all $k \in {\mathcal S}_r$}, so that also $\sum_{k\geq 1} kp_ks_k/\E[D] \geq 1/2$. Hence,
    \eqn{
    F_{\beta}' \left(  \frac{\sum_{k\geq 1} k p_k s_k}{\expec[D]} \right) = -\log f_{\beta} \left(1-  \frac{\sum_{k\geq 1} k p_k s_k}{\expec[D]} \right),
    }
by \eqref{fbp}.  Combining this equation with \eqref{dogk} and (b), we conclude that the maximizers in \eqref{limiting-VP} are solutions of 
   	\eqn{
	\label{smx}
	I'(s_k) +2  B - k \log f_{\beta} \left( 1 - \frac{\cla{\sum_{j\geq 1} j p_j s_j}}{\expec[D]} \right) =0, \qquad \col{k \in {\mathcal S}_r}.
   	}
In \eqref{smx}, we multiply the equation for $s_r$ with $r=\dmin$ by $k$ and substract the equation for \col{$s_k$ with $k\in \mathcal{S}_r \setminus \{r\}$} multiplied by $r$, which gives 
	\eqn{
	k I'(s_r) - r I'(s_k) +2(k-r) B =0.
	}
Note that	
	\eqn{
	I'(s)=\log\Big(\frac{1-s}{s}\Big).
	}
Thus, we arrive at, now writing $s=s_r$,
	\eqn{
	\Big( \frac{1-s}{s}\Big)^k = \left( \frac{1-s_k}{s_k}\right)^r \e^{-2B(k-r)},
	}
which is equivalent to
	\eqn{
	\frac{1-s_k}{s_k} = \left( \frac{1-s}{s}\right)^{k/r} \e^{2(k-r)B/r}.
	}
We conclude that
	\eqn{
	\label{sk-def-b}
	s_k=s_k(s) = \left[\left(\e^{2B} \frac{1-s}{s}\right)^{k/r} \e^{-2B} +1 \right]^{-1}.
	} 
Defining
	\eqn{
	\label{v-def-r}
	w= \left(\e^{2B}\frac{1-s}{s}\right)^{1/r},
	}
%which satisfies $v\in[0,1]$ since $s\geq 1/2$. The latter follows from \eqref{sk-def-b} combined with the fact that $\sum_{k\geq 1} k p_k s_k/\expec[D] \geq 1/2$ precisely when $s\geq 1/2$. 
we thus arrive at
	\eqn{
	\label{sk-def-wk}
	s_k= \left[w^{k} \e^{-2B} +1 \right]^{-1}.
	}
This proves \eqref{sk-def} \col{for all $k \in {\mathcal S}_r$}.

In terms of $w$ and using \cla{\eqref{sk-def-wk}}, \col{equation} \eqref{smx} \col{with $k=r$} becomes
%	\eqn{
%	I'(s) +2B - r\log f_{\beta} \left( 1 - \cla{\frac{\sum_{j\geq 1} j p_j s_j}{\expec[D]}} \right) =0,
%	}
\cla{
	\eqn{
	\log (w^r \e^{-2B}) +2B - r\log f_{\beta} \left( 1 - \cla{\frac{\sum_{j\geq 1} j p_j  \left[w^{j} \e^{-2B} +1 \right]^{-1}}{\expec[D]}} \right) =0,
	}
	}
which is equivalent to 
	\eqn{
	\label{v-equa-a}
	 w= f_{\beta} \left(1 - \E \left(\left[ 1 +  w^{D^\star}\e^{-2B}\right]^{-1} \right) \right),
	 }
where $D^\star$ is the size-biased distribution of $D$ given by 
	\eqn{
	\pp(D^\star=k) = \frac{k p_k}{\col{\expec[D]}}.
	}
Note that $f_{\beta}$ takes values in $(\e^{-2\beta},1]$, so $w\in(\e^{-2\beta},1]$. 
We now compute the inverse function of $f_{\beta}$ as 
%\begin{eqnarray*} \label{ivf}
%f_{\beta}(x) = s & \Leftrightarrow & \e^{-2 \beta} (1-2x)+ \sqrt{1+ (\e^{-4\beta } -1)(1-2x)^2} = 2(1-x)s \notag\\
%& \Leftrightarrow &  1+ (\e^{-4\beta } -1)(1-2x)^2 = 4(1-x)^2s^2 - 4\e^{-2 \beta} (1-2x)(1-x)s+ \e^{-4 \beta}(1-2x)^2 \notag \\
%& \Leftrightarrow &  4x(1-x) = 4(1-x)^2s^2 - 4\e^{-2 \beta} (1-2x)(1-x)s \notag \\
%& \Leftrightarrow &  x = (1-x)s^2 - \e^{-2 \beta} (1-2x)s \notag \\
%& \Leftrightarrow &  x(1+s^2-2\e^{-2 \beta} s) = s^2 - \e^{-2 \beta}s \notag \\
%& \Leftrightarrow &  x = \frac{s^2 - \e^{-2 \beta}s}{(1+s^2-2\e^{-2 \beta} s)}. 
%\end{eqnarray*}
	\eqn{
	f_{\beta}^{-1}(y)= \frac{y^2 - \e^{-2 \beta}y}{1+y^2-2\e^{-2 \beta} y}.
	}
Then \eqref{v-equa-a} becomes 
	\eqn{
 	\frac{1 - \e^{-2 \beta}w }{1+w^2-2\e^{-2 \beta} w } =  \E  \left[ \left(1+w^{D^\star} \e^{-2B} \right)^{-1}\right],
	}
which is \eqref{v-def}.

We can see that if $B >0 $, then $w=1$ is not a solution to the above equation. %Otherwise, when $B=0$, $w=1$ is a solution corresponding to $(s_k)_{k\geq 1} =(\tfrac{1}{2})_{k\geq 1}$, which is not a maximizer as shown in Lemma \ref{lem-crit-v} below. 
Thus, we only need to consider the solution in $(\e^{-2\beta},1)$. 
\qed
\medskip

\noindent
{\it Proof of Theorem \ref{thm-pressure-CM-ann-deter}.} Due to Lemma \ref{lem-sk-sol}, combined with \col{\eqref{varphiand-def}}, we see that $\varphiand(\beta,B)$ is well defined, and thus  \col{\eqref{varphiand-def}} completes the proof.
\qed
\medskip

Below, we will rely on the fact that if we know the solution of \eqref{v-def}, then we can find the annealed pressure by  \col{\eqref{varphiand-def}} and \eqref{sk-def}. This will be used to identify the critical value.

\subsection{Critical value for deterministic degrees: Proof of Theorem \ref{thm-crit-CM-ann-deter}}
\label{sec-crit-value-deter}
We aim to show  that the critical inverse temperature is 
 	\begin{equation} 
	\label{etp}
	\betacand= \betacqud  = \atanh(1/\nu)= \frac{1}{2} \log \left( \frac{\E[D^\star]}{\E[D^\star]-2} \right),
 	\end{equation}
\mo{where  $\nu$  is defined in \eqref{nu-def} and the fact that $\nu= \E[D^\star]-1$ is used.}

We have already shown that 
    $$\varphiand(\beta,B)=\frac{\beta \expec[D]}{2}+\sup\limits_{s_k\in(0,1)}G((s_k)_{k\geq 1}).
    $$
By Lemma \ref{lem-sk-sol},
   	\eqan{
	\label{kmg}
     	\kM_G &= \{(s_k)_{\col{k \in {\mathcal S}_r}} \colon 0\leq s_k \leq 1, (s_k)_{\col{k \in {\mathcal S}_r}} \textrm{ is a maximizer of }G(\cdot) \} \nn\\
      	&\subseteq \{(s_k(w,B))_{\col{k \in {\mathcal S}_r}} \colon \e^{-2 \beta} \leq w \leq 1, w \textrm{ is a solution to } \eqref{v-def}\},
	}
where
	\eqn{
	\label{sk-wB-def}
	s_k(w,B)=\frac{\e^{2B}}{\e^{2B}+w^k}.
	} 
\col{and we recall that ${{\mathcal S}}_r :=\{ i \ge r \,:\, p_i >0\}$ with
$r=\dmin=\min\{i\geq 1\colon p_i>0\}$  the minimal asymptotic degree.} 
%Moreover, the equation $H_{\beta,B}'(w)=0$ is equivalent to    
%	\begin{equation} 
%	\label{rneq}
%	\frac{1 - cw }{1+w^2 -2 cw } =  \E  \left[ \frac{a}{a+w^{D^\star}}  \right],
%	\end{equation}
%with 
%	\eqn{
%	\label{param-def}
%	c=\e^{-2 \beta}, \qquad w=v\e^B.
%	}
%Then 
%	\begin{equation}
%	\label{mag}
%	M(\beta,B)=2\sum_{k\geq 1} p_ks_k^{\star}-1 = 2\E\left[ \frac{a}{a+(w^\star)^{\sss D}}   \right]-1 ,
%	\end{equation}
%where 
%	\eqn{
%	s_k^{\star}=\frac{a}{a+(w^{\star})^k},
%	}
%with $w^{\star} $ a solution of \eqref{rneq} in $(0,1)$. While we expect the solution to \eqref{rneq} to be {\em unique}, we have no proof for this. Instead, we investigate the structure of all possible solutions. For this, let us define 
%	\begin{eqnarray}
%	\underline{w}&=&\inf \{w\in (0,1)\colon w \textrm{ is a maximizer of  } H_{\beta,B}(w)\},\\
%	\overline{w}&=&\sup \{w\in (0,1)\colon w \textrm{ is a maximizer of  } H_{\beta,B}(w)\}.
%	\end{eqnarray}
%We notice that the equation \eqref{rneq} always has a solution $w=0$.  When $w=0$, $s_k(w)=1$ for every $k\geq r=\dmin$, which is not a maximizer of \eqref{limiting-VP} since 	
%	\eqn{
%	G((1)_{k\geq r})=\expec[D]F_\beta(1)=0,
%	}
%while, by \cite[Lemma 3.1 and (2.16)]{Can17a},
%	\eqn{
%	G((1/2)_{k\geq r})=I(1/2)+\expec[D] F_\beta(1/2)= \log 2+\expec[D] F_\beta(1/2)=\log 2.
%	}
%Therefore, in the definition of $\underline{w}, \overline{w}$, we only consider the solutions in $(0,1)$. In the following lemma, we show that the spontaneous magnetization is positive precisely when $\beta>\betacqud$, where $\betacqud$ is defined in \eqref{etp}:

\begin{lemma}[Spontaneous magnetization]
	\label{lem-crit-v} 
%For all $\beta\ge 0$ and $B>0$, let $\Mand(\beta,B)$ denote the annealed spin magnetization defined in \eqref{M-def-qu-ann-d}. 
The spontaneous magnetization defined in \eqref{spon-M-def-qu-ann-d} satisfies that
	\begin{itemize}
		\item[(a)] \mo{$\Msand(\beta) =0$} when $\beta < \betacqud$;
		
		\item[(b)] \mo{$\Msand(\beta) >0$} when $\beta > \betacqud$.
	\end{itemize}
\end{lemma}
%\medskip
%The following corollary is a direct consequence of  Lemma \ref{lem-crit-v}:
%
%\begin{corollary}[Equality critical values for fixed degrees]
%For all degree distributions satisfying Conditions \ref{cond-degrees-regcond}(a)-(b) and (d), $\betacand=\betacqud$.
%\end{corollary}
%\CG{This corollary seems redundant, as it seems to be the same as Th.  \ref{thm-crit-CM-ann-deter}. Can it be removed?}
% 
\noindent 
\col{\it Proof of Lemma \ref{lem-crit-v}}
 We start with part (a). Assume that $\beta < \betacqud$.  Then 
	\eqn{
	c=\e^{-2\beta} > \e^{-2\col{\betacqud}} = \frac{m-2}{m}, \qquad \textrm{ with } \quad m= \E[D^\star]. 
	}
Define 
    $$
    a=\e^{2B}.
    $$
The equation \eqref{v-def} becomes 
	\begin{equation}
	g(w):= \E  \left[ \frac{a}{a+w^{D^\star}}  \right] - \frac{1 - cw }{1+w^2-2c w} =0.
	\end{equation}
We observe that for any $w \in (0,1)$, the function $x \mapsto a/(a+w^x) $ is concave. Therefore, by Jensen's inequality,
	\begin{equation} 
	\label{glh}
	g(w) \leq h_a(w):=  \frac{a}{a+ w^{m}} - \frac{1 - cw }{1+w^2-2c w}.
	\end{equation}
 For $w\in(0,1)$,
	\begin{equation}
	h_a(w)=0  \quad \Leftrightarrow \quad \tilde{h}_a(w): =  c w^m -w^{m-1} +aw -ac =0.
	\end{equation}
We notice that $h_a(0)=0, h_a'(0^+)<0$ and $g(w) \leq h_a(w)$. Hence, 
	\begin{eqnarray*}
	\underline{w} &:= & \inf \{w\in (\e^{-2\beta},1)\colon g(w)=0\} \\
	&\geq&  \inf \{w\in (\e^{-2\beta},1)\colon h_a(w)=0\} =  \inf \{w\in (\e^{-2\beta},1)\colon \tilde{h}_a(w)=0\} =:\underline{w}_a.
	\end{eqnarray*}
%Moreover, as $a\searrow 1$, which is equivalent to $B\searrow 0$, 
%	\begin{equation}
%	\underline{w}_a \rightarrow \inf\{w \in [0,1]\colon \tilde{h}_0(w)=0  \}.
%	\end{equation}
We have 
\begin{eqnarray*}
	\tilde{h}_a'(w)&=& c m w^{m-1} - (m-1)w^{m-2}+a \\
	\tilde{h}_a''(w)&=&(m-1)w^{m-3}(cmw-(m-2)).
\end{eqnarray*}
Since $c> (m-2)/m$, one has $(m-2)/(cm) \in (0,1)$.  Thus, $\tilde{h}_a''(v)$ changes its sign from negative to positive  when $w$ crosses $(m-2)/(cm)$. Hence,   $\tilde{h}_a'(v)$ attains the minimum value at $(m-2)/(cm)$. Using $c> (m-2)/m$,  \mo{ we obtain that 
	\eqn{
	\tilde{h}_a' \left( \tfrac{m-2}{cm} \right) =a-\left( \tfrac{m-2}{cm} \right)^{m-2}
	}
is positive for any $a\geq 1$, therefore  $\tilde{h}_a$ is an increasing function in $[0,1]$,  and $\tilde{h}_a(w)=0$ has a unique solution, say  $\underline{w}_a$}. Moreover, the unique solution $w$ of $\tilde{h}_{1}(w)=0$ is $1$. We can check that $\tfrac{\partial \tilde{h}_1}{\partial w}(1) = (mc-(m-2)) \neq 0$. Therefore, the implicit function theorem gives that $\underline{w}_a$ is a differentiable function in $a$, satisfying that $\underline{w}_a \rightarrow 1$ as $a\rightarrow 1$ and 
	\eqn{
	\label{wa-diff}
	\frac{\partial \underline{w}_a}{\partial a} \Big \arrowvert_{a=1} 
	= - \frac{\tfrac{\partial \tilde{h}_a(w)}{\partial a}  \arrowvert_{(w,a)=(1,1)}}{\tfrac{\partial \tilde{h}_a(w)}{\partial w}  \arrowvert_{(w,a)=(1,1)}} = -\frac{1}{mc-(m-2)}.
	}
Thus, $a\mapsto \underline{w}_a$ is differentiable in $a=1$ with bounded derivative. Further, \colrev{under the assumption that $\beta < \betacqud$, if $B=0$ then the equation \eqref{v-def} has the unique solution $w=1$ (this follows combining the fact that
the unique solution $\underline{w}_1$ of $\tilde{h}_{1}(w)=0$ is $1$ and the fact that the solution $w$ of \eqref{v-def}
is bounded by  $\underline{w}_1 \le w \le 1$)}.
%\CG{How can we conclude that when $B=0$, the equation \eqref{v-def} has unique solution $w=1$ ?}
%\Hao{This follows from the fact that $1\geq \underline{w} \geq \underline{w}_1$ (when $B=0$, $a=1$) and $\underline{w}_1=1$.}
Therefore,  
 	\begin{equation}
 	\varphiand(\beta,0)=\mo{\beta\frac{\mathbb{E}(D)}{2}} + G_0((s_k(1,0))_{\col{k \in {\mathcal S}_r}}),
 	\end{equation}
where $G_0(\cdot)$ denotes the function $G(\cdot)$ when $B=0$.	
 We remark that 
 \mo{
 	\begin{eqnarray} \label{msub}
	 \Mand(\beta) 	 &=& \lim_{B\searrow 0}
	 \frac{\varphiand(\beta,B)-\varphiand(\beta,0)}{B} \nn\\
	 &=&\lim_{B\searrow 0} \frac{G((s_k(w^\star,B))_{{k \in {\mathcal S}_r}})-G_0((s_k(1,0))_{{k \in {\mathcal S}_r}})}{B} \notag  \\
 	&=&\lim_{B\searrow 0} \frac{G((s_k(w^\star,B))_{{k \in {\mathcal S}_r}})-G((s_k(1,B))_{{k \in {\mathcal S}_r}})}{B} \nn\\
 	&&+ \lim_{B\searrow 0} \frac{G((s_k(1,B))_{{k \in {\mathcal S}_r}}))-G_0((s_k(1,0))_{{k \in {\mathcal S}_r}})}{B},
 	\end{eqnarray}
	}
 	with $w^\star$ a solution of  \eqref{v-def}.
 Using the fact that $(s_k(1,0))_{\col{k \in {\mathcal S}_r}}=(1/2)_{\col{k \in {\mathcal S}_r}}$ is a stationary point of $G_0$, we can show that
 	\begin{eqnarray} \label{hsub}
 	&&\lim_{B\searrow 0} \frac{G((s_k(1,B))_{\col{k \in {\mathcal S}_r}}))-G_0((s_k(1,0))_{\col{k \in {\mathcal S}_r}})}{B}  \notag \\
 	&&\qquad = \frac{\partial G\big((1/2)_{\col{k \in {\mathcal S}_r}}\big)}{ \partial B} \Big \arrowvert_{B=0} + \sum_{\col{k \in {\mathcal S}_r}} \frac{\partial G\big((s_k)_{\col{k \in {\mathcal S}_r}}\big)}{ \partial s_k} 
	\Big \arrowvert_{(s_k)_{\col{k \in {\mathcal S}_r}}=(1/2)_{\col{k \in {\mathcal S}_r}}, B=0} 
	\times \frac{\partial s_k}{ \partial B} \Big \arrowvert_{B=0} =0,
 	\end{eqnarray}
\colrev{where we recall that the function} \changed{$G=G_{\beta,B}$ both explicitly depends on the external field $B$, as well as
	implicitly depends on $B$ through the $s_k$ variables.}{}
  By the mean-value theorem, there is some $\tilde{w} \in (w^\star,1)$, such that $G((s_k(w^\star,B))_{\col{k \in {\mathcal S}_r}})-G((s_k(1,B))_{\col{k \in {\mathcal S}_r}}) =G'((s_k(\tilde{w},B))_{\col{k \in {\mathcal S}_r}})(w^\star -1)$. Therefore, using the facts that $\underline{w}_a \leq \underline{w}\leq w^\star \leq \tilde{w} \leq 1$ and that $\underline{w}_a \rightarrow 1$ as $a \searrow 1$ (or equivalently $B \searrow 0$), 
 	\begin{eqnarray} \label{hsb}
	&& \lim_{B\searrow 0}\frac{|G((s_k(w^\star,B))_{\col{k \in {\mathcal S}_r}})-G((s_k(1,B))_{\col{k \in {\mathcal S}_r}})|}{B}=\lim_{B\searrow 0} \frac{|G'((s_k(\tilde{w},B))_{\col{k \in {\mathcal S}_r}})||w^\star -1|}{B} \nn \\
	&&\leq  \lim_{B\searrow 0} \frac{|G'((s_k(\tilde{w},B))_{\col{k \in {\mathcal S}_r}})|(1 - \underline{w})}{B}   \leq  \lim_{B\searrow 0} \frac{|G'((s_k(\tilde{w},B))_{\col{k \in {\mathcal S}_r}})|(1 - \underline{w}_a)}{B}  \nn\\
	&&= \lim_{B\searrow 0} |G'((s_k(\tilde{w},B))_{\col{k \in {\mathcal S}_r}})| \lim_{a\searrow 1 } \frac{(1 - \underline{w}_a)}{a-1} \lim_{B	\searrow 0 } \frac{\e^{2B}-1}{B},
 	\end{eqnarray}
by the chain rule. The right-hand side vanishes by \eqref{wa-diff}, since the first term converges to $0$ since $G'((s_k(1,0))_{\col{k \in {\mathcal S}_r}})=0$ and since the second and third terms are bounded. This, together with \eqref{msub}, shows that \mo{$\Msand(\beta)=0$}, which completes the proof of (a).
\medskip

We now prove  (b). Assume that  $\beta > \betacqud$. We will prove that there exist positive constants $\theta=\theta(\beta)$ and $\delta=\delta(\beta)$, such that, for all \mo{$0 <  B\leq \delta$},
	\eqn{
	\label{mxw}
	\kM_G \subseteq \{(s_k(w,B))_{\col{k \in {\mathcal S}_r}} \colon \e^{-2 \beta} \leq w \leq 1-\theta, w \textrm{ is a solution to } \eqref{v-def}\}.
	}
Let $w^\star \in (\e^{-2\beta}, 1)$ be any solution to $\eqref{v-def}$. Then $(s_k(w^\star,B))_{\col{k \in {\mathcal S}_r}}$ is a stationary point of $G$, and thus
    \eqn{
    G'((s_k(w^\star,B))	= \sum_{\col{k \in {\mathcal S}_r}} \frac{\partial G\big((s_k(w,B))_{\col{k \in {\mathcal S}_r}}\big)}{ \partial s_k} 
    \Big \arrowvert_{w=w^\star}  \times s_k'(w^\star,B)) =0. 
    }
Moreover,  we can compute that
	\eqan{
	G''((s_k(w^\star,B)) &= \sum_{i,\col{k \in {\mathcal S}_r}}\frac{\partial^2 G\big((s_k(w,B))_{\col{k \in {\mathcal S}_r}}\big)}{ \partial s_k \partial s_i} 	\Big \arrowvert_{w=w^\star}  \times s_k'(w^\star,B) s'_i(w^\star,B)\nn\\
	&\qquad + \sum_{\col{k \in {\mathcal S}_r}}\frac{\partial G\big((s_k(w,B))_{\col{k \in {\mathcal S}_r}}\big)}{ \partial s_k } 	\Big \arrowvert_{w=w^\star}  \times s_k''(w^\star, B) \nn  \\
	&=\sum_{i,\col{k \in {\mathcal S}_r}}\frac{\partial^2 G\big((s_k(w,B))_{\col{k \in {\mathcal S}_r}}\big)}{ \partial s_k \partial s_i} 	\Big \arrowvert_{w=w^\star}  \times \mo{s_k'(w^\star,B)} s'_i(w^\star,B) =:H(w^\star,B), 
	}
with 
    	$$ 
    	H(w,B )= \sum_{\col{k \in {\mathcal S}_r}} p_k I''(s_k)(\mo{s'_k(w,B)}) ^2  -  \frac{ f'_{\beta} \left(1- \tfrac{\sum_{k\geq 1} ks_k p_k }{\expec[D]} \right)}
	{f_{\beta} \left(1- \tfrac{\sum_{k\geq 1} k s_k p_k}{\expec[D]} \right)} \frac{1}{\expec[D]}  
    	\left(\sum_{\col{k \in {\mathcal S}_r}} kp_ k s_k'(w,B)\right)^2.
     	$$
When $w\rightarrow 1$ and $B\rightarrow 0$ (or equivalently $a\rightarrow1$), for all $k \in {\mathcal S}_r$,
	\begin{eqnarray}
	s_k(w, B)\rightarrow \tfrac{1}{2}, \quad s_k'(w, B) \rightarrow -\frac{k}{4}, \quad   I''(s_k(w,B)) \rightarrow -4.
	\end{eqnarray}
Moreover, 
	$$
	\frac{f'_{\beta}(1/2)}{f_{\beta}(1/2)} = 2(1-\e^{-2\beta}).
	$$
Therefore, as $w\rightarrow 1$ and $B\rightarrow 0$,
	\begin{eqnarray}
	H(w,B) \rightarrow	H(1,0)&=& - \frac{1}{4}\sum k^2 p_k + \frac{1-\e^{-2 \beta}}{8 \expec[D]} \left(\sum p_kk^2\right)^2\nn\\
	& > & -\frac{1}{4}\E[D^2]  + \frac{2 \E[D]}{8\E[D^2] \E[D]} \E[D^2]^2 = 0, 
	\end{eqnarray}
since, for $\beta>\betacqud$,
	$$
	\e^{-2 \beta} <\e^{-2 \betacqud} = 1 - \frac{2}{\E[D^\star]} = 1- \frac{2 \E[D]}{\E[D^2]}.
	$$
This implies that there exist positive constants $\theta$ and $\delta$, such that $H(w,B)>0$ for all $1-\theta \leq w\leq  1$ and $0 \leq  B \leq \delta$. Therefore, if $w^\star \in (1-\theta, 1)$ is a solution to \eqref{v-def}, then $G'((s_k(w^\star,B)) =0$ and $G''((s_k(w^\star,B))>0$. Hence, $(s_k(w^\star,B))_{k \in {\mathcal S}_r}$ is not a maximizer. Combining this with \eqref{kmg}, we get \eqref{mxw}.

The above implies that if $\mo{0 < B\leq \delta}$, for any maximizer $(s_k(w^\star, B))_{k \in {\mathcal S}_r}$, one has 
	\begin{eqnarray}
	 \sum_{\col{k \in {\mathcal S}_r}}p_ks_k(w^\star)&=&\sum_{\col{k \in {\mathcal S}_r}}p_k\frac{a}{a+(w^\star)^{\mo{k}}}= \E\left(\frac{a}{a+(w^\star)^{\sss \rm D}}\right)\nn\\
 	&\geq& \E\left(\frac{1}{1+(1-\theta)^{\sss \rm D}}\right) =\frac{1}{2}+\eta,
	\end{eqnarray}
for some $\eta >0$. This implies in particular that
    \eqn{
    	\label{mget}
    \kM_G \subseteq  \Big \{(s_k)_{\col{k \in {\mathcal S}_r}} \colon \sum_{\col{k \in {\mathcal S}_r}} p_ks_k \geq  \tfrac{1}{2} +\eta \Big \}.
    }
 On the other hand, we observe that with $\sigma_+=\{v\colon \sigma_v=+1\}$,
	\begin{eqnarray} \label{mnet}
	\mo{M_n^{\rm \sss an,d}}(\beta,B)&:=& \sum_{\sigma} \Big(\frac{2|\sigma _+|}{n}-1\Big) \col{\mu_n^{\rm\sss{an,d}}}(\sigma)\notag \\
	&\geq& \eta \mu_n^{\rm\sss {an, d}} \left(\sigma\colon |\sigma_+| > n(1+\eta)/2\right)-\mu_n^{\rm\sss{an,d}} \left(\sigma\colon |\sigma_+| \leq n(1+\eta)/2\right)\nn\\
	&=&\eta-(\eta+1)\mu_n^{\rm\sss {an,d}} \left(\sigma\colon |\sigma_+| \leq  n(1+\eta)/2\right).
	\end{eqnarray}
Further,
	\eqn{
	\col{\mu_n^{\rm\sss{an,d}}} \left(\sigma\colon |\sigma_+|\leq  n(1+\eta)/2\right)= \frac{\E[Y_{n}(\beta,B)]}{\E[Z_n(\beta,B)]},
	}
with, as in \eqref{eq-partit-function}, 
	\eqn{
	Y_n(\beta,B)=\sum_{\tfrac{|\sigma_+|}{n}\leq \tfrac{1+\eta}{2}}\exp \Big\{\beta \sum_{(i,j) \in E_n} \sigma_i \sigma_j + B\sum_{i \in [n]}\sigma_i\Big\}.
	}
Using the same argument as for the partition function, it follows that 
	\begin{equation*}
	\frac{1}{n}\log \E[Y_n(\beta,B)]=\max_{(s_k)_{\col{k \in {\mathcal S}_r}} \in \kA_{\eta}} G((s_k)_{\col{k \in {\mathcal S}_r}})+o(1),
	\end{equation*}
where 
    	$$
	\kA_{\eta}= \Big\{(s_k)_{\col{k \in {\mathcal S}_r}} \in \kA \colon  \sum_{\col{k \in {\mathcal S}_r}} p_ks_k \leq \tfrac{1+\eta}{2} \Big \}.
	$$
\mo{We conclude that, uniformly in $B>0$
\be
\label{pluto}
\limsup_{n\to\infty} {M_n^{\rm \sss an,d}}(\beta,B)  \ge \eta - (\eta+1)  \lim_{n\to\infty} \exp\Big[- n \Big({\max_{(s_k)_{{k \in {\mathcal S}_r}} \in \kA} G((s_k)_{k \in {\mathcal S}_r}) - \max_{(s_k)_{{k \in {\mathcal S}_r}} \in \kA_{\eta }} G((s_k)_{{k \in {\mathcal S}_r}})\Big)}\Big]. 
\ee
}
Since $\kA_{\eta }$ is  closed in  $\kA$, and $\kA$ is a compact space, also $\kA_{\eta }$ is a compact set. Hence, on $\kA_{\eta}$ the continuous function $G$ attains the maximal value at some point \col{${\mathbf s}_{\eta} \in \kA_{\eta}$}. By  \eqref{mget}, all maximizers of $G$ are not in $\kA_{\eta }$. Therefore,     
		\begin{equation} 
		\label{gxy}
		\max_{(s_k)_{\col{k \in {\mathcal S}_r}} \in \kA} G((s_k)_{k \in {\mathcal S}_r}) - \max_{(s_k)_{\col{k \in {\mathcal S}_r}} \in \kA_{\eta }} G((s_k)_{\col{k \in {\mathcal S}_r}}) 
		= \max_{(s_k)_{k \in {\mathcal S}_r} \in \kA} G((s_k)_{k \in {\mathcal S}_r}) - G(\col{{\mathbf s}_{\eta}}) =\varepsilon >0,
		\end{equation}
for some $\varepsilon$. 
\mo{
As a consequence, \eqref{pluto} implies that uniformly in $B>0$
\be
\label{pluto3}
\limsup_{n\to\infty} M_n^{\rm \sss an,d}(\beta,B)  \ge \eta >0. 
\ee
Since  $(\psinand(\beta,B))_{n\ge 1}$ is a sequence of convex functions in $B\in \mathbb{R}$ converging to $\varphiand(\beta,B)$,
we have (see \cite[Proposition I.3.2 ]{Simo93})
\be
\label{pluto5}
(D_B^+\varphiand)(\beta,B) \ge \limsup_{n\to\infty} (D_B^+\psinand)(\beta,B)
\ee
where $(D_B^+\varphiand)(\beta,B)$ denotes the right derivative w.r.t. B, i.e.
\be
(D_B^+\varphiand)(\beta,B) := \lim_{h\searrow 0} \frac{\varphiand(\beta,B+h) - \varphiand(\beta,B)}{h}. 
\ee
On the other hand, we have differentiability of the pressure at finite volume $\psinand(\beta,B)$
and thus on the right hand side of \eqref{pluto5} we can replace the right derivative $D_B^+$ with the derivative $D_B$, obtaining
\be
(D_B^+\varphiand)(\beta,B) \ge \limsup_{n\to\infty} (D_B\psinand)(\beta,B) = \limsup_{n\to\infty}  M_n^{\rm \sss an,d}(\beta,B).
\ee
By \eqref{pluto3} we thus have
\be
\label{pluto4}
(D_B^+\varphiand)(\beta,B) \ge \eta >0.
\ee
Now for a real convex function $f$ we have continuity from above of the right derivative (see, e.g., \cite[Theorem 1.1.7]{Horm94})
\be
\lim_{h\to 0+} D^+f(x+h) = D^+f(x).
\ee
As a consequence, taking the limit $B\searrow 0$ in \eqref{pluto4}, we conclude that if  $\beta > \betacqud$ then
\be
M^{\rm \sss an,d}(\beta) >0.
\ee 
This completes the proof of part (b).
}

\qed
\medskip

\noindent
{\it Proof of Theorem \ref{thm-crit-CM-ann-deter}.} The proof follows directly from Lemma \ref{lem-crit-v}.
\qed

\section{Proofs for  i.i.d.\  degrees}
\label{sec-proof-iid}
In this section, we investigate the annealed Ising model with i.i.d.\ degrees. In Section \ref{sec-proof-prop-inf-press-iid}, we prove Proposition \ref{prop-inf-press-iid}. 
In Section \ref{sec-part-func-iid-fin}, we identify the partition function and prove Theorem \ref{thm-pressure-CM-ann-iid} for degree distributions with finite supports.
In Section \ref{sec-part-func-iid-inf}, we extend the analysis to general degree distributions. In Section \ref{sec-crit-value-iid}, we identify an upper bound for the
annealed critical value for configuration model with i.i.d.\ degrees and prove Theorem \ref{thm-crit-CM-ann-iid}. 
%\col{We conclude with Section \ref{sec-grg}, where we extend
%the analysis to the generalized random graph with i.i.d. weights.}

\subsection{Finiteness of pressure per particle for \col{i.i.d. degrees}: Proof of Proposition \ref{prop-inf-press-iid}}
\label{sec-proof-prop-inf-press-iid}
We first note that
	\eqan{
	\expec[Z_n(\beta, {B})]&\geq \expec\Big[\exp \Big\{\beta \sum_{(i,j) \in E_n} 1+ B \sum_{i \in [n]} 1\Big\}\Big]
	= \expec\Big[\exp \Big\{\beta |E_n|+ n B\Big\}\Big]\\
	&=\e^{n B}\expec\Big[\e^{\beta\sum_{i\in[n]} D_i/2}\Big].\nn
	}
The latter is finite precisely when $\expec[\e^{\beta D/2}]<\infty$, as required.

Further, we can similarly bound from above
	\eqan{
	\label{Zn-bd-fin}
	\expec[Z_n(\beta, {B})]&\leq 2^n \expec\Big[\exp \Big\{\beta \sum_{(i,j) \in E_n} 1+ |B|\sum_{i \in [n]} 1\Big\}\Big]\\
	&= 2^n \e^{n|B|}\expec\Big[\exp \Big\{\beta |E_n|\Big\}\Big]\nn\\
	&=2^n \e^{n|B|}\expec\Big[\e^{\beta\sum_{i\in[n]} D_i/2}\Big].\nn
	}
Again, this is finite precisely when $\expec[\e^{\beta D/2}]<\infty$. The latter bound also implies that $\limsup_{n\rightarrow \infty} \col{\psinaniid(\beta,B)}<\infty$ when $\expec[\e^{\beta D/2}]<\infty$, which completes the proof.
\qed

\subsection{Identification of the partition function for  i.i.d.\ degrees: Proof of Theorem \ref{thm-pressure-CM-ann-iid} for finite supports}
\label{sec-part-func-iid-fin}
In this section, we prove Theorem \ref{thm-pressure-CM-ann-iid} for degree distributions having a finite support. We will extend the proof to general degree distributions satisfying $\expec[\e^{\beta D/2}]<\infty$ below. Let us consider the annealed Ising model on the configuration model with i.i.d.\ sequence of degrees given by the distribution $\boldsymbol{p}$, i.e., the sequence $(D_i)_{i\in[n]}$ are i.i.d.\ random variables with distribution $\pp(D=i)=p_i$, \cla{with $p_i=0$ if $i>i_0$, for a given $i_0$ (finite support condition).}  For any probability distribution $\boldsymbol{q} =(q_k)_{k\geq 1}$, let $\varphiand(\beta,B;\boldsymbol{q})$ be the annealed pressure of the Ising model on the configuration model with degree distribution $\boldsymbol{q}$.

Let $\boldsymbol{P}^{\sss(n)}$ be the empirical degree distribution, i.e., 
	\eqn{
	P_k^{\sss{(n)}}=\frac{1}{n}\sum_{i\in[n]} \indic{D_i=k}.
	}
Then
	\eqn{
	\E \left[Z_n (\beta,B) \right]= \E \left( \e^{n \psinand(\beta,B;\boldsymbol{P}^{\sss(n)})} \right).
	}
We note that $\psinand(\beta,B;\boldsymbol{p})=\varphiand(\beta,B;\boldsymbol{p})+o(1)$ uniformly in $\boldsymbol{p}$, since $\boldsymbol{p}$ only takes finitely many values.
\color{black} See in particular \eqref{psinand-1}, which gives a uniform bound for $\boldsymbol{p}$ only taking finitely many values, and \eqref{limiting-VP}, which shows that the error in the continuum approximation of the values of $(s_k)_{k\in \mathcal{S}_r}$ over which we maximize is uniform as well. \color{black}
By Varadhan's Lemma,
	\eqn{
	-\lim_{n\rightarrow \infty} \frac{1}{n}\log\E \left[Z_n (\beta,B) \right]= \inf\limits_{\boldsymbol{q}} [H(\boldsymbol{q} \mid \boldsymbol{p}) -  \varphiand(\beta,B;\boldsymbol{q})],
	}
where we define the relative entropy 
	\eqn{
	H(\boldsymbol{q} \mid {\boldsymbol{p}}) =\sum_{i\ge 1} q_i\log{(q_i/ p_i)}
	}
and we use that the functional $\boldsymbol{q}\mapsto \varphiand(\beta,B;\boldsymbol{q})$ is bounded and continuous, \color{black}since, e.g.
	\eqn{
	\varphiand(\beta,B;\boldsymbol{q})-\varphiand(\beta,B;\boldsymbol{q}')
	\leq \sup_{(s_k)_{k \in {\mathcal S}_r}}[G_{\beta,B}((s_k)_{k \in {\mathcal S}_r};\boldsymbol{q})-G_{\beta,B}((s_k)_{k \in {\mathcal S}_r};\boldsymbol{q}')],
	}
%\col{We note that $H(\boldsymbol{q} \mid {\boldsymbol{p}})$ is well defined since both $\boldsymbol{q}$ and  $\boldsymbol{p}$ only takes finitely many values.}
where now we also make the dependence on the degree distribution explicit in the notation. Since $\boldsymbol{q}\mapsto G_{\beta,B}((s_k)_{k \in {\mathcal S}_r};\boldsymbol{q})$ is uniformly continuous, the upper continuity of $\boldsymbol{q}\mapsto \varphiand(\beta,B;\boldsymbol{q})$ follows. The lower continuity can be proved in an identical manner.\color{black} We conclude that
	\eqn{
	\label{iid}
	\varphianiid(\beta,B) = \sup\limits_{\boldsymbol{q}} [\varphiand(\beta,B;\boldsymbol{q}) -   H(\boldsymbol{q} \mid \boldsymbol{p}) ].
	}
This identifies the pressure of the annealed Ising model on the configuration model with i.i.d.\ degrees, and thus proves Theorem \ref{thm-pressure-CM-ann-iid} for degree distributions having finite support. 
\qed

\subsection{Identification of the partition function for  i.i.d.\ degrees: Proof of Theorem \ref{thm-pressure-CM-ann-iid} for infinite supports}
\label{sec-part-func-iid-inf}
When $\expec[\e^{\beta D/2}]=\infty$, we know that the pressure per particle is infinite. In this section, we investigate what happens otherwise, by extending the analysis in the previous section to such degree distributions \cla{with infinite support} satisfying $\expec[\e^{\beta D/2}]<\infty$. We proceed by upper and lower bounds.
\medskip

\noindent
{\it Proof of lower bound in Theorem \ref{thm-pressure-CM-ann-iid} for infinite supports.}
We start with the lower bound, which is the easiest. Fix $k$ large, to be determined later on. We use that
	\eqan{
	\E \left[Z_n (\beta,B) \right]&\geq \E \left( \e^{n \psinand(\beta,B;\boldsymbol{P}^{\sss(n)})} \indic{P_l^{\sss(n)}=0~\forall l\geq k}\right)\nn\\
	&=\Big(\sum_{i=1}^{k-1} p_i\Big)^n \E \left[\e^{n \psinand(\beta,B;\boldsymbol{P}^{\sss(n)})}\mid D_i<k~ \forall i\in[n]\right].
	}
Fix $\vep>0$, and use that $\Big(\sum_{i=1}^{k-1} p_i\Big)^n\geq \e^{-\vep n}$  for $k\geq k(\vep)$ sufficiently large. It thus suffices to investigate the conditional expectation.
Here, we use that, conditionally on $D_i<k~ \forall i\in[n]$, the sequence $(D_i)_{i\in[n]}$ is i.i.d.\ with probability mass function $\tilde{p}_i=p_i/\sum_{l=1}^{k-1} p_l$, \cla{$i=1,\ldots, k-1$}. Thus, we can use the analysis in Section \ref{sec-part-func-iid-fin} to arrive at
	\eqn{
	\liminf_{n\rightarrow \infty}\frac{1}{n}\log \E \left[Z_n (\beta,B) \right]
	\geq  \col{\sup\limits_{\boldsymbol{q}} [\varphiand(\beta,B;\boldsymbol{q})-H(\boldsymbol{q} \mid \widetilde{\boldsymbol{p}})]} \cla{-\vep}.
	}
Since
	\eqan{
	H(\boldsymbol{q} \mid \widetilde{\boldsymbol{p}})&=\sum_{i=1}^{k-1} q_i\log{(q_i/\tilde p_i)}
	=\sum_{i=1}^{k-1} q_i\log{(q_i/p_i)} -\log\Big(\sum_{i=1}^{k-1} p_i\Big)\sum_{i=1}^{k-1} q_i\nn\\
	&=\sum_{i=1}^{k-1} q_i\log{(q_i/p_i)}+o(1),
	}
we arrive at
	\eqn{
	\label{iid-inf-lb}
	\liminf_{k\rightarrow \infty} \liminf_{n\rightarrow \infty}\frac{1}{n}\log \E \left[Z_n (\beta,B) \right]
	\geq  \col{\sup\limits_{\boldsymbol{q}} [\varphiand(\beta,B;\boldsymbol{q})-H(\boldsymbol{q} \mid \boldsymbol{p})]-\vep},
	}
with $H(\boldsymbol{q} \mid \boldsymbol{p})=\sum_{i=1}^{\infty} q_i\log{(q_i/p_i)}$ and the supremum is over all probability distributions on $\mathbb{N}$. This establishes the lower bound.
\qed
\medskip

\noindent
{\it Proof of upper bound in Theorem \ref{thm-pressure-CM-ann-iid} for infinite supports.} The upper bound is substantially more involved, and will follow from three lemmas.
We start by giving an alternative representation of the pressure per particle. 

It has been shown in \cite[(7.2) and (7.3)]{Can17a} that
   	\eqan{
   	\label{zncb}
   	\E[Z_n(\beta,B)] = \e^{O(1/n)}\sum_{j=0}^n \binom{n}{j} \e^{B(2j-n)} \E\left(\e^{\ell_n[\beta/2 +F_{\beta}(\ell_j/\ell_n)]}\right),
   	}
where for all $j\leq n$, 
   	\eqn{
	\ell_j=D_1+\cdots+D_j
	}
consists of the partial sums of the degrees. Our first lemma investigates some properties of the function $F_\beta$:

\begin{lemma}[\col{Properties of $F_\beta$}]
\label{labcd}
	Let $a,b,c,d$ be non-negative real numbers satisfying $a\leq b,c \leq d$ and $b-a \leq d-c$. Then
	\eqn{
	\label{fabcd}
	dF_{\beta}(c/d)-bF_{\beta}(a/b) \leq 0.
    }
\end{lemma}

\begin{proof}
Recall that $F_{\beta}$ satisfies
    \eqan{
    F_{\beta}(t)&=F_{\beta}(1-t) \leq 0 \,\, \forall t \in [0,1], \label{sfbt} \\
    F_{\beta}'(t) &\leq 0 \,\, \forall t \in [0,\tfrac{1}{2}], \qquad  F_{\beta}''(t) \geq 0 \,\, \forall t \in [0,1].	 
    \label{cfbt}
    }
We  notice  that it is sufficient to prove the lemma for the case $\tfrac{a}{b}, \tfrac{c}{d} \leq \tfrac{1}{2}$. Indeed, let $\bar{a}=b-a$ and $\bar{c}=d-c$. Then $\bar{a} \leq \bar{c}$, and thus  $(a\wedge \bar{a}) \leq b, (c\wedge \bar{c}) \leq d$. By \eqref{sfbt}, $F(a/b)=F((a\wedge \bar{a})/b)$  and $F(c/d)=F((c\wedge \bar{c})/d)$. Hence, we only need to prove \eqref{fabcd} for the case that $\tfrac{a}{b}, \tfrac{c}{d} \leq \tfrac{1}{2}$.

Assume that $\tfrac{a}{b}, \tfrac{c}{d} \leq \tfrac{1}{2}$. We show  that for all $t\in [0,1]$,
   \eqn{
   	\label{xlo}
	x(t)=F_{\beta}(t)-tF_{\beta}'(t) \leq 0.	
    }
Indeed, using \eqref{sfbt} and \eqref{cfbt}, $x'(t)=-tF_{\beta}''(t)\leq 0$ for all $t\in[0,1]$. Thus $x(t) \leq x(0)=0$ for all $t\in [0,1]$. 

Since $F_{\beta}''(t) \geq 0$, by a Taylor expansion, we obtain that $F_{\beta}(x)-F_{\beta}(y) \leq F_{\beta}'(x)(x-y)$ for all $0\leq x, y \leq 1$. Hence, 
    \eqan{
    dF_{\beta}(c/d) - bF_{\beta}(a/b) &= (d-b) F_{\beta}(c/d) + b(F_{\beta}(c/d)-F_{\beta}(a/b)) \nn \\
    &\leq (d-b) F_{\beta}(c/d) + bF_{\beta}'(c/d) \left(\frac{c}{d}-\frac{a}{b}\right).
    }
Observe that, since $a\leq c$, 
   \eqan{
   b\left(\frac{c}{d}-\frac{a}{b}\right)&=\frac{cb-ad}{d} \geq \frac{cb -cd}{d}=-(d-b)\frac{c}{d}.
   }
Further $F_{\beta}'(c/d) \leq 0$, since $c/d \leq 1/2$. Hence, using the last two inequalities, we get 
   \eqan{
   dF_{\beta}(c/d) - bF_{\beta}(a/b) \leq (d-b) \Big (F_{\beta}(c/d) -(c/d)F_{\beta}'(c/d) \Big ) \leq 0,
   }
by \eqref{xlo}, as required.
\end{proof}
\medskip

We use Lemma \ref{labcd} to truncate the degree distribution, as we explain now. For any $k\geq 1$, we consider a truncated sequence $(D_i^{\sss(k)})_{i\geq 1}$ defined by $D_i^{\sss(k)}=D_i\wedge k$.  We define correspondingly the partition function $Z^{\sss(k)}_n(\beta,B)$ and $\ell_j^{\sss(k)}=D_1^{\sss(k)} + \cdots +D_j^{\sss(k)}$ for $j\leq n$. Then we also have 
  	\eqan{
  	\E[Z_n^{\sss(k)}(\beta,B)] \col{=} \e^{\kO(1/n)}\sum_{j=0}^n \binom{n}{j} \e^{B(2j-n)} \E\left(\e^{\ell_n^{\sss(k)}[\beta/2 +F_{\beta}(\ell_j^{\sss(k)}/\ell_n^{\sss(k)})]}\right).
  	}
Moreover, it is clear that for all $j\leq n$, 
  	$$
   	\ell_j^{\sss(k)} \leq  \ell_n^{\sss(k)}, \ell_j \leq \ell_n, \quad \textrm{and} \quad \ell_n^{\sss(k)} - \ell_j^{\sss(k)} =\sum_{i=j+1}^n (D_i \wedge k ) \leq \sum_{i=j+1}^n D_i = \ell_n-\ell_j. 
  	$$
Therefore, applying Lemma \ref{labcd}, we obtain that, for all $j\leq n$,
%  	$$
%  	\ell_n^{\sss(k)}F_{\beta}(\ell_j^{\sss(k)}/\ell_n^{\sss(k)}) \leq   \ell_nF_{\beta}(\ell_j/\ell_n),
%  	$$
\cla{	$$
  	\ell_nF_{\beta}(\ell_j/\ell_n)  \leq   \ell_n^{\sss(k)}F_{\beta}(\ell_j^{\sss(k)}/\ell_n^{\sss(k)}),
  	$$}
which implies, using \eqref{zncb}, that 
  	\eqan{
  	\E[Z_n(\beta,B)] \leq \e^{O(1/n)}\sum_{j=0}^n \binom{n}{j} \e^{B(2j-n)} \E\left(\e^{\ell_n^{\sss(k)}[\beta/2 +F_{\beta}(\ell_j^{\sss(k)}/\ell_n^{\sss(k)})]} 
	\e^{\beta(\ell_n-\ell_n^{\sss(k)})/2}\right).
  	}

Define
	\eqn{
	N_{\sss \geq k}=\sum_{i\in[n]} \indic{D_i\geq k}.
	}
Then, by conditioning on $N_{\sss \geq k}$, we obtain
	  \eqan{
  	\E[Z_n(\beta,B)] \leq \e^{O(1/n)}\sum_{m,j=0}^n \binom{n}{j} \e^{B(2j-n)} \E\left[\e^{\ell_n^{\sss(k)}[\beta/2 +F_{\beta}(\ell_j^{\sss(k)}/\ell_n^{\sss(k)})]} 
	\e^{\beta(\ell_n-\ell_n^{\sss(k)})/2}\mid N_{\sss \geq k}=m\right]\prob(N_{\sss \geq k}=m).
  	}
Conditionally on $N_{\sss \geq k}=m$, the vectors $(D_i\wedge k)_{i\in [n]}$ and $\ell_n-\ell_n^{\sss(k)}=\sum_{i\in [n]} (D_i-k)\indic{D_i\geq k}$ are independent, so that
	 \eqan{
  	\E[Z_n(\beta,B)] &\leq \e^{O(1/n)}\sum_{m,j=0}^n \binom{n}{j} \e^{B(2j-n)} \E\left[\e^{\ell_n^{\sss(k)}[\beta/2 +F_{\beta}(\ell_j^{\sss(k)}/\ell_n^{\sss(k)})]} \mid N_{\sss \geq k}=m\right]\nn\\
	&\qquad\qquad\qquad\times \E\left[\e^{\beta(\ell_n-\ell_n^{\sss(k)})/2}\mid N_{\sss \geq k}=m\right]\prob(N_{\sss \geq k}=m).
  	}
We compute
	\eqn{
	\E\left[\e^{\beta(\ell_n-\ell_n^{\sss(k)})/2}\mid N_{\sss \geq k}=m\right]=\E[\e^{\beta (D-k)/2}\mid D\geq k]^m.
	}
In the following lemma, we investigate the other conditional expectation:

\begin{lemma}[Conditional pressure per particle]
\label{lem-cond-pres}
For any $q_{\sss \geq k}\in [0,1]$ and $k\geq 1$,
	\eqan{
	&\lim_{n\rightarrow \infty}\frac{1}{n} \log \sum_{j=0}^n \binom{n}{j} \e^{B(2j-n)}  
	\E\left[\e^{\ell_n^{\sss(k)}[\beta/2 +F_{\beta}(\ell_j^{\sss(k)}/\ell_n^{\sss(k)})]} \mid N_{\sss \geq k}=\lceil n q_{\sss \geq k}\rceil\right]\nn\\
	&\qquad=\sup_{\widetilde{\boldsymbol{q}}} \big[\varphiand(\beta,B; \widetilde{\boldsymbol{q}})-\widetilde{H}(\widetilde{\boldsymbol{q}}\mid \widetilde{\boldsymbol{p}})\big],
	}
where $\widetilde{\boldsymbol{p}}=(p_1, \ldots, p_{k-1}, p_{\sss \geq k})$ and $\widetilde{\boldsymbol{q}}=(q_1, \ldots, q_{k-1}, q_{\sss \geq k})$, and
	\eqn{
	\widetilde{H}(\widetilde{\boldsymbol{q}}\mid \widetilde{\boldsymbol{p}})
	=\sum_{i=1}^{k-1} q_i\log(q_i/p_i) -(1-q_{\sss \geq k})\log\big((1-q_{\sss \geq k})/(1-p_{\sss \geq k})\big).
	}
\end{lemma}

\proof This follows immediately from Varadhan's lemma, as in \eqref{iid}, noting that the vector $(D_i\wedge k)_{i\in[n]}$ satisfies a conditional large deviation principle with rate function $\widetilde{H}(\widetilde{\boldsymbol{q}}\mid \widetilde{\boldsymbol{p}})$. Indeed, conditionally on  $N_{\sss \geq k}=\lceil n q_{\sss \geq k}\rceil$, this sequence consists of $q_{\sss \geq k}n$ values $k$, and the remaining $(1-q_{\sss \geq k})n$ values are i.i.d.\ with distribution $\widetilde{\boldsymbol{p}}$. The rate function of this sequence can be computed using the ratio of probabilities as
	\eqan{
	\widetilde{H}(\widetilde{\boldsymbol{q}}\mid \widetilde{\boldsymbol{p}})
	&=-\lim_{n\rightarrow \infty} \frac{1}{n} \log\prob(N_1=q_1 n, \ldots, N_{k-1}=q_{k-1} n, N_{\sss \geq k}=q_{\sss \geq k}n)
	+\lim_{n\rightarrow \infty} \frac{1}{n} \log \prob(N_{\sss \geq k}=q_{\sss \geq k}n)\nn\\
	&=\Big[\sum_{i=1}^{k-1} q_i\log(q_i/p_i)+q_{\sss \geq k}\log(q_{\sss \geq k}/p_{\sss \geq k})\Big]\nn\\
	&\qquad \qquad-\Big[q_{\sss \geq k}\log(q_{\sss \geq k}/p_{\sss \geq k})
	+(1-q_{\sss \geq k})\log\big((1-q_{\sss \geq k})/(1-p_{\sss \geq k})\big)\Big],
	}
and the $q_{\sss \geq k}\log(q_{\sss \geq k}/p_{\sss \geq k})$ terms cancel, as required.	
\qed
\medskip

We conclude that
	\eqan{
  	\limsup_{n\rightarrow \infty} \frac{1}{n}\log\E[Z_n(\beta,B)] &\leq \sup_{\widetilde{\boldsymbol{q}}}\col{\Big[
	\varphiand(\beta,B; \widetilde{\boldsymbol{q}})-\widetilde H(\widetilde{\boldsymbol{q}}\mid \widetilde{\boldsymbol{p}})+q_{\sss \geq k}\log \E[\e^{\beta (D-k)/2}\mid D\geq k]}\nn\\
	&\col{\qquad\qquad -q_{\sss \geq k}\log(q_{\sss \geq k}/p_{\sss \geq k})-(1-q_{\sss \geq k})\log((1-q_{\sss \geq k})/(1-p_{\sss \geq k}))\Big]}\nn\\
	&= \sup_{\widetilde{\boldsymbol{q}}}\big[
	\varphiand(\beta,B; \widetilde{\boldsymbol{q}})-\sum_{i=1}^{k-1} q_i\log(p_i/q_i)-q_{\sss \geq k}\log(q_{\sss \geq k})\nn\\
	&\qquad\qquad+q_{\sss \geq k}\log \E[\e^{\beta (D-k)/2}\indic{D\geq k}]\big].
  	}
When $k$ is sufficiently large $\E[\e^{\beta (D-k)/2}\indic{D\geq k}]\leq 1$ by dominated convergence, so that we arrive at
	\eqan{
  	\cla{\limsup_{n\rightarrow \infty}} \frac{1}{n}\log\E[Z_n(\beta,B)] 
	&\leq \sup_{\widetilde{\boldsymbol{q}}}\big[\varphiand(\beta,B; \widetilde{\boldsymbol{q}})-\sum_{i=1}^{k-1} q_i\log(p_i/q_i)-q_{\sss \geq k}\log(q_{\sss \geq k})\big].
  	}
We aim to let $k\rightarrow \infty$ in the above expression. We see that, as \col{$k\to\infty$},
	\eqn{
	\sum_{i=1}^{k-1} q_i\log(p_i/q_i)-q_{\sss \geq k}\log(q_{\sss \geq k})\rightarrow \sum_{i=1}^{\infty} q_i\log(p_i/q_i)
	}
for every proper probability mass function $\boldsymbol{q}$. What remains to prove is that the optimal $\boldsymbol{q}$ does not put mass in infinity:

\begin{lemma}[Optimizer does not put mass at infinity]
For any $\vep>0$, 
	\eqn{
	\liminf_{k\rightarrow \infty}\liminf_{n\rightarrow \infty} \frac{1}{n}\log\E[Z_n(\beta,B)\indic{\sum_{i\in[n]} \indic{D_i\geq k}\geq \vep n}] =-\infty.
	}
\end{lemma}

\proof We bound as in \eqref{Zn-bd-fin},
	\eqan{
	\E\Big[Z_n(\beta,B) \indic{\sum_{i\in[n]} \indic{D_i\geq k}\geq \vep n}\Big]
	&\leq 2^n \e^{n|B|}\expec\Big[\e^{\beta\sum_{i\in[n]} D_i/2}\indic{\sum_{i\in[n]} \indic{D_i\geq k}\geq \vep n}\Big]\nn\\
	&=2^n \e^{n|B|} \sum_{m\geq \vep n} {\binom{n}{m}} p_{\sss \geq k}^m (1-p_{\sss \geq k})^{n-m}
	\expec\Big[\e^{\beta\sum_{i\in[n]} D_i/2}\mid N_{\sss \geq k}=m\Big].
	}
We compute
	\eqn{
	\expec\Big[\e^{\beta\sum_{i\in[n]} D_i/2}\mid N_{\sss \geq k}=m\Big]=
	\expec\Big[\e^{\beta D/2}\mid D\geq k\Big]^m\expec\Big[\e^{\beta D/2}\mid D<k\Big]^{n-m}.
	}
This yields
	\eqan{
	\E\Big[Z_n(\beta,B) \indic{\sum_{i\in[n]} \indic{D_i\geq k}\geq \vep n}\Big]
	&\leq 2^n \e^{n|B|} \sum_{m\geq \vep n} {\binom{n}{m}} \expec\big[\e^{\beta D/2}\indic{D\geq k}\big]^m\expec\big[\e^{\beta D/2}\indic{D<k}\big]^{n-m}\nn\\
	&=  2^n \e^{n|B|}\expec\big[\e^{\beta D/2}\big]^n \sum_{m\geq \vep n} {\binom{n}{m}} r_k^m (1-r_k)^{n-m},
	}
where
	\eqn{
	r_k=\frac{\expec\big[\e^{\beta D/2}\indic{D\geq k}\big]}{\expec\big[\e^{\beta D/2}\big]}.
	}
Since $\expec\big[\e^{\beta D/2}\big]<\infty$, we have that $\lim_{k\rightarrow \infty} r_k=0$, so that 
	\eqan{
	\limsup_{n\rightarrow \infty}\frac{1}{n}\log\sum_{m\geq \vep n} {\binom{n}{m}} r_k^m (1-r_k)^{n-m}
	&=\limsup_{n\rightarrow \infty}\frac{1}{n}\log\prob({\sf Bin}(n,r_k)\geq \vep n)\nn\\
	&=-\vep \log(\vep/r_k)-(1-\vep)\log((1-\vep)/(1-r_k))\nn\\
	&\rightarrow -\infty,
	}
for any $\vep>0$ as $k\rightarrow \infty.$
\qed
\medskip

We arrive at
	\eqan{
  	\lim_{n\rightarrow \infty} \frac{1}{n}\log\E[Z_n(\beta,B)] 
	&\leq \sup_{\boldsymbol{q}}\big[\varphiand(\beta,B; \boldsymbol{q})-\sum_{i=1}^{\infty} q_i\log(p_i/q_i)\big]\nn\\
	&=\sup_{\boldsymbol{q}}\big[\varphiand(\beta,B; \boldsymbol{q})-H(\boldsymbol{q}\mid \boldsymbol{p})\big].
  	}
This proves the required upper bound.
\qed
\medskip

Together with the corresponding lower bound proved in \eqref{iid-inf-lb}, we arrive at
	\eqn{
	\label{iid-inf}
	\lim_{n\rightarrow \infty}\frac{1}{n}\log \E \left[Z_n (\beta,B) \right]
	=\col{\sup\limits_{\boldsymbol{q}} [\varphiand(\beta,B;\boldsymbol{q})-H(\boldsymbol{q} \mid \boldsymbol{p})]},
	}
which establishes the formula for the annealed pressure for i.i.d.\ degrees having infinite support.
\qed

\subsection{Annealed critical value for configuration model with i.i.d.\ degrees: Proof of Theorem \ref{thm-crit-CM-ann-iid}}
\label{sec-crit-value-iid}
We first notice that 
	\eqn{ \label{pher}
	\varphianiid(\beta,B) = \sup\limits_{w, \boldsymbol{q}} R_{\beta,B}(w, \boldsymbol{q}),
	}
where
	\eqn{
	\label{Rwq}
	R_{\beta, B}(w,  \boldsymbol{q}) = -H( \boldsymbol{q}\mid  \boldsymbol{p}) + r_{\beta,B}(w, \bq)
	}
and
 	\eqn{  
	\label{rwq}
	r_{\beta,B}(w, \bq) = \frac{\beta \E[D( \boldsymbol{q})]}{2}+ G_{\beta, B}((s_k(w,B))_{k\geq 1};  \boldsymbol{q}).
	}
Here we explicitly write the parameters $(\beta, B, \bq)$ of all functions involved. We claim that $R_{\beta, B}(\cdot)$ is a continuous function in a compact space $[0,1] \times \kP$, with $\kP$ the space of discrete probability measures. Thus, there exists a non-empty set of maximizers. 

In the proof below, we fix $\beta < \betacaniid$.  For $B>0$, define 
	\eqn{
	\bar{w}(B) = \sup \{w^{\star}\colon \textrm{ there exists $\bq^{\star}$ such that $(w^{\star}, \bq^{\star})$ is a maximizer of $R_{\beta, B}(\cdot)$} \}.
	}
The proof consists of four steps.
\medskip

\paragraph{\bf Step 1: An optimizer $w^\star(B)$ that converges to 1.} We first show that $\bar{w}(0^+) = \limsup_{B \rightarrow 0^+} \bar{w}(B)=1$ by contradiction. Assume that $\bar{w}(0^+) <1$. Then there exists $\delta, \varepsilon >0$, such that  when $0<B<\delta$, we have $w^{\star}=w^{\star}(B)< 1 -\varepsilon$ for all  maximizers $(w^{\star}, \bq^{\star})$ of $R_{\beta, B}$. Since $[0,1] \times \kP$ is compact, as $B\searrow 0$, we can take a subsequence of $(w^{\star}, \bq^{\star})$ converging to a limit $(\tilde{w}, \tilde{\bq})$ with $\tilde{w}\leq 1 -\varepsilon$. Therefore, 
	\eqn{
	\varphianiid(\beta, 0^+)= R_{\beta, 0}(\tilde{w}, \tilde{\bq}) = -H( \tilde{\bq}\mid  \boldsymbol{p}) + \frac{\beta \E[D( \tilde{\bq})]}{2}+ G_{\beta, 0}((s_k(\tilde{w}, \mo{0}))_{k\geq 1};  \tilde{\bq}).
	}
On the other hand,
	\eqn{
	\varphianiid(\beta, B) \geq  -H( \tilde{\bq}\mid  \boldsymbol{p}) + \frac{\beta \E[D( \tilde{\bq})]}{2}+ G_{\beta, B}((s_k(\tilde{w},B))_{k\geq 1};  \tilde{\bq}).
	}
Thus,
	\begin{eqnarray*}
 %	\Maniid(\beta,0^+)\Msandiid(\beta) 
	\mo{\Msaniid(\beta)}&=&
%	\liminf_{B\searrow 0}
         \mo{\lim_{B\searrow 0}}\frac{\varphianiid (\beta,B) - \varphianiid (\beta, 0)}{B}\\
 	&\mo{\geq} &\lim_{B\searrow 0} \frac{G_{\beta, B}((s_k(\tilde{w},B))_{k\geq 1};  \tilde{\bq}) - G_{\beta, 0}((s_k(\tilde{w},0))_{k\geq 1};  \tilde{\bq})}{B}\\
	&=&\lim_{B\searrow 0} \frac{G_{\beta, B}((s_k(\tilde{w},B))_{k\geq 1};  \tilde{\bq}) - G_{\beta, B}((s_k(\tilde{w},0))_{k\geq 1};  \tilde{\bq})}{B}\\
	&&\qquad+\lim_{B\searrow 0} \frac{G_{\beta, B}((s_k(\tilde{w},0))_{k\geq 1};  \tilde{\bq}) - G_{\beta, 0}((s_k(\tilde{w},0))_{k\geq 1};  \tilde{\bq})}{B}.
	\end{eqnarray*}
The first term converges to
	\eqn{
	\sum_{k\geq 1} \frac{\partial}{\partial s_k} G_{\beta, 0}((s_k(\tilde{w},B))_{k\geq 1})\big|_{B=0} \frac{\partial}{\partial B} s_k(\tilde{w},B)=0,
	}
since, by assumption, $(s_k(\tilde{w},0))_{k\geq 1}$ is a stationary point of $G_{\beta, 0}((s_k)_{k\geq 1})$. The second term converges to $\sum_{k\geq 1} \tilde{q}_k (2s_k(\tilde{w},0) -1)$, so that
	\begin{eqnarray*}
%	\Maniid(\beta,0^+)
	\mo{\Msaniid(\beta)} &\mo{\ge}&\sum_{k\geq 1} \tilde{q}_k (2s_k(\tilde{w},0) -1)= 2 \sum_{k\geq 1} \frac{ \tilde{q}_k}{1+\tilde{w}^k} -1\\
	&=&\E \left[\frac{2}{1+w^{\tilde{D}}}\right] -1=\E \left[\frac{1-w^{\tilde{D}}}{1+w^{\tilde{D}}}\right]\\
	&\geq&\E \left[\frac{1-(1-\varepsilon)^{\tilde{D}}}{1+(1-\varepsilon)^{\tilde{D}}}\right] >0.
	\end{eqnarray*}
Here we have used that $\tilde{w} \leq 1-\varepsilon$ and $\tilde{D}$- the random variable with law $\tilde{\bq}$ is not $0$ almost surely. By \eqref{def_beta_c} applied to $\Maniid(\beta)$ and $\betacaniid$, this contradicts the assumption that $\beta < \betacaniid$.
\medskip

\paragraph{\bf Step 2: A Lagrange multiplier equation.} 
We have now proved that $\bar{w}(0^+)=1$, and thus we can assume that $w^{\star}(B) \rightarrow 1$ as $B \searrow 0$ for one sequence of maximizers $(w^{\star}(B))_{B\geq 0}$. Now we analyze the limit of $\bq^{\star}=\bq^{\star}(B)$ as $B \searrow 0$. Since $(w^{\star}(B), \bq^{\star}(B))$ is a maximizer of $R_{\beta,B}$, $(w^{\star}(B), \bq^{\star}(B))$ satisfies
	\eqn{
	\label{iid-VP}
	\frac{\partial R_{\beta,B}(w^{\star}(B), \bq^\star(B))}{\partial q_i} =\lambda,
	}
for some Lagrange multiplier $\lambda$ that is due to the restriction that $\boldsymbol{q}^\star(B)$ is a probability measure, and all $i$ for which $p_i>0$. 
If $\boldsymbol{p}$  only takes finitely many values, the above system indeed provides the critical points of $R_{\beta,B}$. 

The above Lagrange equation is standard when the degree distribution has finite support, but requires some extra arguments when $\boldsymbol{p}$ has infinite support. In this paragraph, we fix $(\beta,B)$ are omit them from the notation, as we aim to show that \eqref{iid-VP} holds for all $(\beta,B)$. We now provide these arguments. First, we prove that the optimizer $\boldsymbol{q}^\star$ satisfies that $q^\star_i>0$ whenever $p_i>0$. For this, we fix $i\geq 1$ for which $p_i>0$ and a $j\geq 1$ for which $q^\star_j>0$. Assume that $q^\star_i=0$, and consider, instead, the probability distribution $\boldsymbol{q}(\vep)=\boldsymbol{q}^\star+\vep \delta_{i}-\vep \delta_{j}$, where $\delta_i$ is the Kronecker-delta on $i\in \N$.  Since $(w^{\star}, \boldsymbol{q}^\star)$ is the maximizer, we must have that 
	\eqn{
	H(\boldsymbol{q}(\vep) \mid \boldsymbol{p}) - r(w^{\star},\bq(\varepsilon)) \leq H(\boldsymbol{q}^\star \mid \boldsymbol{p}) -  r(w^\star,\boldsymbol{q}^\star).
	}
By the identification of $r$, the function the function $q_i \mapsto r(w^{\star},\boldsymbol{q})$ is differentiable. 
However,
	\eqn{
	H(\boldsymbol{q}(\vep) \mid \boldsymbol{p})-H(\boldsymbol{q}^\star \mid \boldsymbol{p}) 
	=\vep \log(\vep/p_i)+(q_j^\star-\vep) \log((q_j^\star-\vep)/p_j)-q_j^\star \log((q_j^\star-\vep)/p_j),
	}
which is $\vep \log(\vep)(1+o(1)),$ and which has a derivative $-\infty$ at $\vep=0$. This contradicts with the fact that $\boldsymbol{q}^\star$ is the maximizer. We conclude that the optimizer $\boldsymbol{q}^\star$ satisfies that $q^\star_i>0$ whenever $p_i>0$. 
\medskip

\medskip
\noindent

We continue by proving that \eqref{iid-VP} holds. Fix $i,j$ for which $p_i,p_j>0$. By the previous argument, also $q^\star_i, q^\star_j>0$. Pick $\vep$ so small that $\boldsymbol{q}(\vep)=\boldsymbol{q}^\star+\vep \delta_{i}-\vep \delta_{j}$ is a probability measure. Then, by the fact that $(w^{\star}, \boldsymbol{q}^\star)$ is the maximizer of $H(\boldsymbol{q} \mid \boldsymbol{p}) - r(w, \bq)$, we obtain that
	\eqn{
	\label{iid-VP-a}
	\frac{\partial}{\partial \vep} [H(\boldsymbol{q}(\vep) \mid \boldsymbol{p}) -  r(w^{\star},\boldsymbol{q}(\vep))]_{\vep=0}=0,
	}
which is equivalent to
	\eqn{
	\label{iid-VP-b}
	\frac{\partial}{\partial q_i} \Big[H(\boldsymbol{q} \mid \boldsymbol{p}) -  r(w^{\star},\boldsymbol{q})\Big]_{\boldsymbol{q}=\boldsymbol{q}^\star}
	=\frac{\partial}{\partial q_j} \Big[H(\boldsymbol{q} \mid \boldsymbol{p}) -  r(w^{\star},\boldsymbol{q})\Big]_{\boldsymbol{q}=\boldsymbol{q}^\star}.
	}
Since this is true for all $i,j$ for which $p_i,p_j>0$, we get \eqref{iid-VP}.
\medskip

\paragraph{\bf Step 3: Identification of a limit of the optimizer $\boldsymbol{q}^\star(B)$.} In this step, we show that, along an appropriate subsequence, $\boldsymbol{q}^\star(B)$ converges to an explicit limit. Note that $\frac{\partial}{\partial q_i} H(\boldsymbol{q}\mid \boldsymbol{p})=\log(q_i/p_i)+1$. We can thus alternatively write \eqref{iid-VP} as
	\eqn{
	\label{iid-VP-c}
	\log(q_i^\star(B)/p_i)-\frac{\partial}{\partial q_i} r_{\beta,B}(w^{\star}(B),\boldsymbol{q})\big|_{\boldsymbol{q}=\boldsymbol{q}^\star(B)}=\lambda.
	}
By \eqref{rwq},
	\eqan{
	\frac{\partial}{\partial q_i} r_{\beta,B}(w^{\star}(B),\boldsymbol{q})&= \frac{\partial}{\partial q_i} [\beta \expec[D(\boldsymbol{q})]/2+G_{\beta, B}\big((s_k(w^{\star}))_{k\geq 1}; \boldsymbol{q})].
	}
 Further, by \eqref{G(sk)-def},
	\eqn{\label{Gsq-def}
	G_{\beta,B}((s_k)_{k\geq 1}; \boldsymbol{q})=\sum_{k\geq 1} q_k I(s_k) + B \Big(2\sum_{k \geq 1} s_kq_k -1\Big)
	+ \expec[D(\boldsymbol{q})] F_{\beta}\Big(\frac{\sum_{k\geq 1} k q_k s_k}{ \expec[D(\boldsymbol{q})] }\Big),
	}
Thus,
	\eqan{
	\label{def-stationary-eq}
	\frac{\partial}{\partial q_i} G_{\beta,B}\big((s_k)_{k\geq 1}; \boldsymbol{q})
	&=I(s_i)+2s_iB+i F_{\beta}\Big(\frac{\sum_{k\geq 1} k q_k s_k}{ \expec[D(\boldsymbol{q})] }\Big)\nn\\
	&\qquad +\expec[D(\boldsymbol{q})]F_{\beta}'\Big(\frac{\sum_{k\geq 1} k q_k s_k}{ \expec[D(\boldsymbol{q})] }\Big)
	\Big(\frac{is_i}{\expec[D(\boldsymbol{q})]}-\frac{i\sum_{k\geq 1} k q_k s_k}{\expec[D(\boldsymbol{q})]^2}\Big).
	}
Take $B\searrow 0$. Extract a subsequence along which $\bq^\star(B)$ converges to a limit, say $\bq(\beta)$ and $w^\star(B) \rightarrow 1$. Along this subsequence, $(s_k(w^\star(B)))_{k\geq 1} \rightarrow (1/2)_{k\geq 1}$.  Then, using that $F_{\beta}(\tfrac{1}{2})= \log\big(\tfrac{1}{2}[1+\e^{-2\beta}]\big)$, \eqref{def-stationary-eq} becomes 
	\eqan{
	\frac{\partial}{\partial q_i} G_{\beta, 0}\big((s_k(1))_{k\geq 1}; \boldsymbol{q}(\beta))
	&= \log 2+ i F_{\beta}(\tfrac{1}{2})= \log 2 +\frac{i}{2} \log\Big(\tfrac{1}{2}[1+\e^{-2\beta}]\Big).
	}	
We conclude that \eqref{iid-VP} becomes
	\eqn{
	\log(q_i(\beta)/p_i)+\beta i/2+\frac{i}{2} \log\Big(\tfrac{1}{2}[1+\e^{-2\beta}]\Big)=\log(q_i(\beta)/p_i) + \frac{i}{2}\log\cosh(\beta)=\lambda,
	}
so that
	\eqn{
	\label{qi-beta-def}
	q_i(\beta)=p_i \cosh(\beta)^{i/2}/c(\beta),
	}
where $c(\beta)=\expec[\cosh(\beta)^{D/2}]<\infty$, since $\expec[\e^{\beta D/2}]<\infty$. 
\medskip

\paragraph{\bf Step 4: Completion of the proof of Theorem \ref{thm-crit-CM-ann-iid}.} 
For all $B\geq 0$, we have 
	\eqn{
	\varphianiid(\beta,B) \geq -H(\bq(\beta) \mid \bp) + \varphiand(\beta,B;\boldsymbol{q}(\beta)).
	}
Moreover, since $w^\star(B) \rightarrow 1$ and $\bq^\star(B)\rightarrow \boldsymbol{q}(\beta)$,
	\begin{eqnarray}
	\varphianiid(\beta,0^+) = \lim_{B\searrow 0} R(w^\star, \bq ^\star) = R(1,\bq(\beta)) &=& -H(\bq(\beta) \mid \bp) +r(1,\bq(\beta))  \notag \\
	& \leq & - H(\bq(\beta) \mid \bp) + \varphiand(\beta,0; \boldsymbol{q}(\beta)).
	\end{eqnarray}
\mo{
Combining the last two inequalities, we obtain, using the definition 
%of $M^{\textrm{an,d}}(\beta, 0^+, \bq(\beta))$ 
%in Lemma \ref{lem-crit-v},
of the spontaneous magnetization
%	\begin{eqnarray}
%	M^{\textrm{an,D}}(\beta, 0^+)=\liminf_{B \searrow 0} \frac{\varphianiid(\beta,B) -\varphianiid(\beta,0)}{B} &\geq &\liminf_{B\searrow 0} \frac{\varphiand(\beta, B) -\varphiand (\beta, 0)}{B}  \notag \\
%	& =&  M^{\textrm{an,d}}(\beta, 0^+; \bq(\beta)).
%	\end{eqnarray}
	\begin{eqnarray}
	M^{\textrm{an,D}}(\beta)=\lim_{B \searrow 0} \frac{\varphianiid(\beta,B) -\varphianiid(\beta,0)}{B} &\geq &\lim_{B\searrow 0} \frac{\varphiand(\beta, B) -\varphiand (\beta, 0)}{B}  \notag \\
	& =&  M^{\textrm{an,d}}(\beta; \bq(\beta)).
	\end{eqnarray}	
Since $\beta < \betacaniid$,  
%$M^{\textrm{an,D}}(\beta, 0^+) =0$. 
$M^{\textrm{an,D}}(\beta) =0$.
Since 
%$M^{\textrm{an,d}}(\beta, 0^+; \bq(\beta))\geq 0$ 
$M^{\textrm{an,d}}(\beta; \bq(\beta))\geq 0$ 
%by the remark below \eqref{M-def-qu-ann-d}, 
we conclude that 
%$M^{\textrm{an,d}}(\beta, 0^+; \bq(\beta))=0$. 
$M^{\textrm{an,d}}(\beta; \bq(\beta))=0$. 
}This is equivalent to  
	\eqn{
	\beta\leq \atanh(1/\nu(\boldsymbol{q}(\beta))),
	}
where we write
	\eqn{\label{nuq}
	\nu(\boldsymbol{q})=\frac{\sum_{k\geq 1}k(k-1)q_k}{\sum_{k\geq 1} kq_k}.
	}
We also notice that $\beta\mapsto \atanh(1/\nu(\boldsymbol{q}(\beta)))$ is monotonically decreasing. Hence the equation $\beta =  \atanh(1/\nu(\boldsymbol{q}(\beta)))$ has a unique solution, denoted by $\bar{\beta}_c^{\sss D}$. Moreover, if $\beta\leq \atanh(1/\nu(\boldsymbol{q}(\beta)))$ then $\beta \leq \bar{\beta}^{\sss \rm D}_c$.  Thus we have proved that if $\beta < \betacaniid$, then $\beta \leq \bar{\beta}^{\sss \rm D}_c$. Therefore, $\betacaniid \leq \bar{\beta}^{\sss \rm D}_c$. Furthermore, for any $\beta\geq 0$, 
	\eqn{
	\nu(\boldsymbol{q}(\beta))\geq \nu,
	}
since
	\eqn{
	\nu(\boldsymbol{q}(\beta))=\frac{\expec[(D^\star-1)\cosh(\beta)^{D^\star/2}]}{\expec[\cosh(\beta)^{D^\star/2}]}\geq \expec[D^\star-1]=\nu,
	}
with strict inequality unless $D^\star$ is constant or $\beta=0$. Thus, we obtain that when $\bar{\beta}^{\sss \rm D}_c>0$ and the degrees are not a.s.\ constant,
	\eqn{
	\betacaniid \leq \bar{\beta}^{\sss \rm D}_c=\atanh(1/\nu(\boldsymbol{q}(\bar{\beta}^{\sss \rm D}_c)))<\atanh(1/\nu)=\beta_c^{\rm qu}.
	}
This completes the proof of Theorem \ref{thm-crit-CM-ann-iid}.
\qed

\medskip
\begin{remark}[Equality $\betacaniid = \bar{\beta}^{\sss \rm D}_c$]
\label{rem-eq-crit-betas-iid-rep}
{\rm We believe that $\betacaniid = \bar{\beta}^{\sss \rm D}_c$, but lack a proof of this fact. Indeed, we expect that the magnetization $\Maniid(\beta,B)\approx 0$ when $(\beta,B)\approx (\betacaniid, 0)$ with $\beta\geq \betacaniid, B\geq 0$. This, in turn, would mean that there exists an optimizer $w^\star(\beta,B)$ that is close to 1, so that there also exists an optimizer $(s^\star_k(\beta,B))_{k\geq 1}$ that is close to $\tfrac{1}{2}$. As a result, there should also be an optimizer $\boldsymbol{q}_k(\beta,B)$ that is close to $\boldsymbol{q}(\beta),$ which should thus be close to critical. This should mean that $\betacaniid =\atanh(1/\nu(\boldsymbol{q}(\betacaniid)))$. Our proof, however, does not yield this, since we have insufficient control over the optimizers.}
\end{remark}

\begin{remark}[Critical behavior]
\label{rem-crit-iid}
{\rm Assume that $\betacaniid = \bar{\beta}^{\sss \rm D}_c$, as we argue above, as well as $\expec[\e^{\betacaniid D/2}]<\infty$. Then, the solution $q_i(\beta)$ in \eqref{qi-beta-def} has exponential tails, since $\cosh(\beta)<\e^{\beta}$. Thus, by \eqref{def-stationary-eq}, one can expect the solution $\boldsymbol{q}(\beta,B)$ to have exponential tails as well for $B>0$ small enough and $\beta>\betacaniid$ with $\beta-\betacaniid$ sufficiently small. As a result, one can expect that the critical behavior is related to that of an annealed system with a degree distribution with exponential tails, which should be equal to those for the Curie-Weiss model. Since this is true for {\em all} degree distributions with $\expec[\e^{\betacaniid D/2}]<\infty$, this suggests that the power-law universality class does not arise for i.i.d.\ degrees, which is quite surprising.
}
\end{remark}
\medskip

\begin{example}[Poisson degrees]
\label{rem-example-Poisson}
{\rm Let us compute the critical inverse temperature when $D$ is the Poisson distribution with intensity $\lambda >0$, that is $\prob (D=i)=p_i$, $i\ge 0$, with 
%\RvdH{Polish this part!}  for all $i\geq 0$, and $\beta >0$,
   \eqan{
  p_i = \e^{-\lambda} \frac{\lambda^i}{i!}\, .
    }   
    From \eqref{qi-beta-def} we get the optimizer for all $\beta>0$:
   \eqan{
 q_i(\beta)= \frac{p_i\cosh(\beta)^{i/2}}{c(\beta)}=\frac{1}{\e^{\lambda}c(\beta)}\frac{(\lambda \sqrt{\cosh(\beta)})^i}{i!}.
    }
Hence, the expected forward degree corresponding to distribution $\boldsymbol{q}(\beta)=(q_i(\beta))_i$ is
   \eqan{\label{eq-nu-q}
   \nu(\boldsymbol{q}(\beta))=\frac{\sum_{i\geq 0} i(i-1)q_i(\beta)}{\sum_{i\geq 0} iq_i(\beta)}=\frac{\sum_{i\geq 2} (\lambda \sqrt{\cosh(\beta)})^i/(i-2)!}{\sum_{i\geq 1} (\lambda \sqrt{\cosh(\beta)})^i/(i-1)!}= \lambda \sqrt{\cosh(\beta)}.
   }
The upper bound $\bar{\beta}^{\sss \rm D}_c$ satisfies \eqref{bcand}, or equivalently
  \eqan{
%  &\quad \tanh(\betacaniid)=1/\nu(\boldsymbol{q}(\betacaniid)) \nn\\ 
  \quad \tanh(\bar{\beta}^{\sss \rm D}_c)= \Big(\lambda \sqrt{\cosh(\bar{\beta}^{\sss \rm D}_c)}\Big)^{-1}, \nn
%  \Leftrightarrow & \quad \cosh(\betacaniid) = \lambda^2 \sinh(\betacaniid)^2= \lambda^2(\cosh(\betacaniid)^2-1) \nn\\
%  \Leftrightarrow & \quad \cosh(\betacaniid) = \frac{1+\sqrt{1+4 \lambda^4}}{2 \lambda^2} %\Leftrightarrow e^{ \betacaniid} = \frac{1}{2} \left[1+\sqrt{1+4 \lambda^4}+\sqrt{4 \lambda^4 -2 +2 \sqrt{1+4 \lambda^4}}\right]
%  \nn \\
%  \Leftrightarrow & \quad  \betacaniid = -\log (2 \lambda^2) + \log   \left[1+\sqrt{1+4 \lambda^4}+\sqrt{2 +2 \sqrt{1+4 \lambda^4}}\right].
     } 
that can be rewritten as
\eqan{
 \quad \cosh(\bar{\beta}^{\sss \rm D}_c) = \lambda^2(\cosh(\bar{\beta}^{\sss \rm D}_c)^2-1).
}
Solving this equation for $\cosh(\bar{\beta}^{\sss \rm D}_c)$ and inverting the hyperbolic cosine, we obtain
	\eqan{
 	\betacaniid\leq \bar{\beta}^{\sss \rm D}_c = -\log (2 \lambda^2) + \log   \left[1+\sqrt{1+4 \lambda^4}+\sqrt{2 +2 \sqrt{1+4 \lambda^4}}\right].
	}

On the other hand, the quenched critical value is  $\betacqud = \atanh(1/\nu(\boldsymbol{p}))$, where $\nu(\boldsymbol{p)}=\lambda$ (as can be readily seen by setting $\beta=0$ in  \eqref{eq-nu-q}), that is
	\eqn{
	\betacqud= \tfrac{1}{2} \log \Big(\tfrac{\lambda+1}{\lambda-1}\Big).
	} 

%$\tfrac{1}{2} \log \Big(\tfrac{\nu(\boldsymbol{p})+1}{\nu(\boldsymbol{p})-1}\Big) 
%= \tfrac{1}{2} \log \Big(\tfrac{\lambda+1}{\lambda
%	-1}\Big)$. 
Let us observe that, while $\betacaniid$ exists for any $\lambda>0$, the quenched critical value is finite only if $\lambda >1$, and as $\lambda \searrow 1$, $\betacqud\rightarrow \infty$, while $ \betacaniid \leq \bar{\beta}^{\sss \rm D}_c \rightarrow \log [(1+\sqrt{5}+\sqrt{2+2\sqrt{5}})/2]$.
}\end{example}

\begin{example}[Geometric degrees]
\label{rem-example-Geo}
{\rm We next investigate the case of geometric degrees, i.e., we assume that $\prob(D=i)=p_i$ with
	\be
	p_i=(1-p)^{i-1}p,\quad  i\geq 1,
	\ee
for some $0<p<1$. 

\changed{As a first remark, to make sure
that the annealed pressure is finite, we use Prop. \ref{prop-inf-press-iid}, and thus require}{}
\colrev{\be
\mathbb{E}[\e^{\beta D/2}] < \infty,
\ee
which amounts to 
\be
\label{condizione}
\beta< -2 \log(1-p).
\ee
 }
Then, by \eqref{qi-beta-def},
	\be
	q_i(\beta)= \frac{(1-p)^{i-1} p (\sqrt{\cosh(\beta)})^i}{c(\beta)},\quad i\geq 1.
	\ee
Hence,
	\eqan{
	\sum_{i\ge 1}i  q_i(\beta)&= \frac{p \sqrt{\cosh(\beta)} }{c(\beta)} \sum_{i\ge 1} i ((1-p) \sqchb)^{i-1}\nn\\
	&= \frac{p \sqrt{\cosh(\beta)} }{c(\beta)}  \left (1-(1-p)\sqchb \right)^{-2},
	}
\colrev{where the summability condition $(1-p)\sqchb < 1$ is guaranteed by the assumption \eqref{condizione}.
%provided that the summability condition $(1-p)\sqchb < 1$ is \colrev {also} satisfied. Under the same assumption, also
Furthermore,}
	\be
	\sum_{i\ge 1}i (i-1) q_i(\beta) = \frac{p\sqchb}{c(\beta)} \frac{2(1-p)\sqchb}{\left (1-(1-p)\sqchb \right)^3}.
	\ee
Then, from \eqref{nuq}, we obtain the expected forward degree 
	\be
	\nu({\bf q}(\beta)) = \frac{2(1-p)\sqchb}{1-(1-p)\sqchb}.
	\ee
Thus, in this case, the equation for  $\bar{\beta}^{\sss \rm D}_c$ in \eqref{bcand} reads as
	\be
	\label{eq-beta-geom}
	\tanh(\bar{\beta}^{\sss \rm D}_c)=\left ( 2(1-p)\sqrt{\cosh(\bar{\beta}^{\sss \rm D}_c)}) \right)^{-1} -\frac 1 2.
	\ee
This equation looks similar to the corresponding equation for the Poisson case, but is in fact quite different.
Setting,  for the sake of notation, $x=\sqrt{\cosh(\bar{\beta}^{\sss \rm D}_c)}$ and $r=1-p$, we can rewrite  \eqref{eq-beta-geom} as
	\be
	\frac{\sqrt{x^4-1}}{x^2}=\frac{1-r x}{2 r x}
	\ee
that, after some manipulation, can be transformed into
	\be
	x^2(3r^2 x^4 +2 r x^3 -x^2 -4r^2)=0.
	\ee
Recalling the definition of  $x$, we conclude that the equation for the critical  $\bar{\beta}^{\sss \rm D}_c$ is
	\be
	\label{equu}
	3r^2 x^4 +2 r x^3 -x^2 -4r^2=0, \quad \mbox{with}\quad x \ge 1.
	\ee
The polynomial  $f_r(x)=3r^2 x^4 +2 r x^3 -x^2 -4r^2$ (with $r>0$) has a local negative maximum at  $x=0$ and  for $x>0$ a unique negative local minimum  at $x_+(r)=\frac{1}{4 r}(\frac{\sqrt{33}}{3}-1)$. Since $f_r^\prime(x)>0$ for $x>x_+(r)$, we conclude that there exists a unique positive solution $x^\star(r) \in (x_+(r),\infty)$ to the equation $f_r(x)=0$. Now we have to check that $x^\star(r)$ corresponds to an admissible solution to \eqref{equu}, that is,  $1\le x^\star(r)<r^{-1}$ for all $0<r\le 1$, i.e., for all $0\le p < 1$.  Observe that we can bound $f_r(x)$ for $x<1$ as
	\be
	f_r(x)= r^2(3x^4-4)+2r x^3 -x^2 < r^2(3x^4-4)+ x^2 (2r-1),
	\ee
and, since $3 x^4-4<-1$ and $x^2<1$,
	\be
	 r^2(3x^4-4)+ x^2 (2r-1) < -r^2+2 r -1 = -(1-r)^2\equiv -p^2.
	\ee
Thus, $f_r(x)<0$ if $x<1$ (with $f_r(1)=0$ if and only if  $r=1$)  and this shows that $x^\star(r) \ge 1$. 
Now we show that $x^*(r)<r^{-1}$ by contradiction. Dropping the dependence of $x^\star(r)$ on $r$, we write  
	\be
	f_r(x^\star)=3 (r x^\star)^2 {x^\star}^2 + 2(rx^\star) {x^\star}^2-{x^\star}^2-4 r^2,
	\ee
and assuming  $x^\star>1$ we obtain,
	\be
	f_r(x^\star) > 4 {x^\star}^2-4 r^2 >0,
	\ee
 (where the last inequality follows from the fact that $x^\star>1$ and $r<1$) in contradiction with the fact that $x^\star$ is the root of $f_r(x^\star)=0$ . The previous inequality proves the claim.
% Therefore we can conclude that for any $0\le p < 1$ there exists a unique  solution to the equation \eqre.f{equu}, that defines the critical  point  $\betacaniid$ and that can be written explicitly as:
% \be
% 3(1-p)^2 \cosh^2(\betacaniid) +2 (1-p)  \cosh^{\frac 3 2}(\betacaniid) -  \cosh (\betacaniid) - 4 (1-p)^2=0. 
% \ee
Therefore we conclude that, for any $0< p < 1$, the critical inverse temperature with \changed{i.i.d.\ }{} geometric degrees is bounded from above by
	\be
	\bar{\beta}^{\sss \rm D}_c= \ln \left ( {x^\star(p)}^2 + \sqrt{{x^\star(p)}^4-1}\right) ,
	\ee
where $x^\star(p)$ is the unique solution to the fourth order equation
	\be\label{eqn-geom-eqn}
	3 (1-p)^2 x^4 + 2 (1-p) x^3 - x^2 -4 (1-p)^2=0.
	\ee
On the other hand, the quenched critical value with geometric degrees is by \eqref{betac-an-d} equal to
 	\be
 	\betacqud= \atanh \left ( \frac{p}{2(1-p)}\right ).
 	\ee
Then, we have  $\betacqud=\infty$ for $p\ge \frac 2 3$, while in the annealed case with i.i.d.\ degrees  there is a finite critical value for any $0<p<1$.
}
%\Gib{To be continued}
\end{example}

\medskip
\paragraph{\bf Acknowledgement.} The work of RvdH is supported by the Netherlands Organisation for Scientific Research (NWO) through VICI grant 639.033.806 and the Gravitation {\sc Networks} grant 024.002.003. \col{C. Giardin\`a and C. Giberti acknowledge financial supports from Fondo di Ateneo per la Ricerca, Universit\`a di Modena e Reggio Emilia.} The work of V. H. Can is supported by  the fellowship no. 17F17319 of the Japan Society for the Promotion of Science, and by the  Vietnam National Foundation for Science and Technology Development (NAFOSTED) under grant number 101.03--2019.310.

\end{document}